\documentclass[11pt]{amsart}

\usepackage[T1]{fontenc}
\usepackage[dvipsnames]{xcolor}
\usepackage[latin1]{inputenc}
\usepackage{orcidlink}
\usepackage[english]{babel}
\usepackage{mathabx}
\usepackage{amsmath,amssymb,amsthm}
\usepackage{etoolbox}
\usepackage{enumitem}
\usepackage{hyperref}

\newtheorem{theorem}{Theorem}[section]
\newtheorem{lemma}[theorem]{Lemma}
\newtheorem{proposition}[theorem]{Proposition}
\newtheorem{corollary}[theorem]{Corollary}

\theoremstyle{definition}
\newtheorem{definition}[theorem]{Definition}

\newtheorem{remark}[theorem]{Remark}
\newtheorem{example}[theorem]{Example}
\newtheorem*{remark*}{Remark}

\newcommand{\Mult}{\operatorname{Mult}}
\renewcommand{\Re}{\operatorname{Re}}
\newcommand{\ccap}{\operatorname{cap}}
\newcommand{\bnd}{{\partial \mathbb{B}_d}}
\DeclareMathOperator{\supp}{supp}
\newcommand{\PHDA}{h^2_d}
\newcommand{\Kor}{Kor\'anyi}

\renewcommand{\MR}[1]{}

\title[Potential theory in the Drury--Arveson space]{Potential theory and boundary behavior in the Drury--Arveson space}
\author[N. Chalmoukis]{Nikolaos Chalmoukis \orcidlink{0000-0001-5210-8206}}
\address{Dipartimento di Matematica e Applicazioni, Universit\`a degli
Studi di Milano Bicocca, Via R. Cozzi 55,  20125, Milano, Italy}
\email{nikolaos.chalmoukis@unimib.it}

\author[M. Hartz]{Michael Hartz \orcidlink{0000-0001-6509-9062}}
\address{Fachrichtung Mathematik, Universit\"at des Saarlandes, 66123 Saarbr\"ucken, Germany}
\email{hartz@math.uni-sb.de}
\thanks{The first author is a member of the INdAM group GNAMPA. Most of this work was carried out when the first author was an Alexander von Humboldt postdoc fellow at Saarland University. The second author was partially supported by the Emmy Noether Program of the German Research Foundation (DFG Grant 466012782)}
\date{\today}


\begin{document}

\begin{abstract}
    We develop a notion of capacity for the Drury--Arveson space $H^2_d$ of holomorphic
    functions on the Euclidean unit ball. We show
    that every function in $H^2_d$ has a non-tangential limit (in fact Kor\'anyi limit) at every point
    in the sphere outside of a set of capacity zero.
    Moreover, we prove that the capacity zero condition is sharp,
    and that it is equivalent to being totally null for $H^2_d$.
    We also provide applications to cyclicity.
    Finally, we discuss generalizations of these results to other function
    spaces on the ball.
\end{abstract}

\maketitle

\section{Introduction}

Let $\mathbb{D}$ denote the open unit disc in the complex plane,
let $\mathcal{O}(\mathbb{D})$ be the algebra of holomorphic functions on $\mathbb{D}$ and let
\begin{equation*}
  H^2(\mathbb{D}) = \Big\{ f \in \mathcal{O}(\mathbb{D}): \sup_{0 \le r < 1} \int_{0}^{2 \pi} |f(r e^{it})|^2 \frac{dt}{2 \pi} < \infty \Big\}
\end{equation*}
be the Hardy space.
A famous theorem of Fatou \cite[Theorem 1.2]{D70} shows that for every function $f \in H^2(\mathbb{D})$, there exists
a Borel set $E \subset \partial \mathbb{D}$ of linear Lebesgue measure $0$ such that
$f$ has a finite non-tangential limit (and in particular radial limit) at each point $\zeta \in \partial \mathbb{D} \setminus E$.
The condition of being Lebesgue null is optimal in the sense that for every compact Lebesgue null
set $E \subset  \partial \mathbb{D}$, there exists $f \in H^2(\mathbb{D})$ with
$\lim_{r \nearrow 1} |f(r \zeta)| = \infty$ for all $\zeta \in E$.

Let $\mathbb{B}_d \subset \mathbb{C}^d$ denote the open Euclidean unit ball.
One way of generalizing $H^2(\mathbb{D})$ to the ball is to replace the normalized arclength
measure on $\partial \mathbb{D}$ with the normalized surface measure $\sigma$ on $\partial \mathbb{B}_d$;
this leads to the Hardy space $H^2(\mathbb{B}_d)$ on the ball.
In this setting, \Kor\ \cite{K69} proved that every function in $H^2(\mathbb{B}_d)$
has a \Kor\ limit at $\sigma$-almost every point in $\partial \mathbb{B}_d$.
The precise definition of \Kor\ limit will be recalled in Section \ref{sec:weak_type}; in particular, the existence of the \Kor\ limit
implies the existence of the non-tangential and hence of the radial limit.

It is by now well understood that for many purposes, a more appropriate generalization of $H^2(\mathbb{D})$
to the ball is the Drury--Arveson space $H^2_d$. This is the reproducing kernel Hilbert space on $\mathbb{B}_d$
whose reproducing kernel is
\begin{equation*}
  \frac{1}{1 - \langle z,w \rangle }.
\end{equation*}
For instance, $H^2_d$ plays key roles in multivariable operator theory \cite{Arveson1998} and in the theory of complete Pick spaces \cite{Agler2000}.
See also \cite{Hartz22a,Shalit15} for an introduction to $H^2_d$.
A central goal of the present article is to establish a version of Fatou's theorem for the Drury--Arveson space.

Since $H^2_d$ is contained in $H^2(\mathbb{B}_d)$, \Kor's theorem in particular
applies to functions in $H^2_d$. However, it is easy to see that functions in $H^2_d$ exhibit
regular boundary behavior that is not predicted by \Kor's theorem.
For instance, if $f \in H^2_d$, then for all $\zeta \in \partial \mathbb{B}_d$, the function
\begin{equation*}
  f_\zeta: \mathbb{D} \to \mathbb{C}, \quad z \mapsto f(z \zeta),
\end{equation*}
belongs to $H^2(\mathbb{D})$, so by Fatou's theorem, the radial limit $\lim_{r \nearrow 1} f(r e^{i t} \zeta)$
exists for almost every $t \in [0, 2 \pi]$. On the other hand,
if $d \ge 2$, then the circle $\{e^{ i t} \zeta: t \in [0,2 \pi]\}$ is $\sigma$-null.

Functions in $H^2_d$ enjoy additional smoothness compared to functions in $H^2(\mathbb{B}_d)$.
The situation is somewhat analogous to that of the classical Dirichlet space $\mathcal{D}$,
which is contained in $H^2(\mathbb{D})$. While Fatou's theorem applies in particular
to functions in $\mathcal{D}$, much more can be said.
Beurling \cite{B40} proved that every function in $\mathcal{D}$ has a non-tangential limit outside
of a set of logarithmic capacity zero.
We will recall the definition of logarithmic capacity below; in particular,
logarithmic capacity zero implies Lebesgue null, but not vice versa.
Carleson \cite{Carleson1952} proved that logarithmic capacity zero is optimal in the sense
that for every compact subset $E \subset \partial \mathbb{D}$ of logarithmic capacity zero,
there exists $f \in \mathcal{D}$ with $\lim_{r \nearrow 1} |f(r \zeta)| = \infty$ for all $\zeta \in E$.

Taking a cue from the boundary results in the Dirichlet space, we are led
to developing a potential theory and hence a notion of capacity that is appropriate
for the Drury--Arveson space. For certain Hardy--Sobolev spaces on $\mathbb{B}_d$,
this was accomplished by Ahern--Cohn \cite{AC89} and Cohn--Verbitsky \cite{Cohn1995}, but the Drury--Arveson
space is an edge case in those results, for which only partial information is available.
Indeed, the main obstacle to applying standard potential theory in the Drury--Arveson space is that the real part and the modulus of the reproducing kernel are not comparable. {The difficulties that the absence of this property entails can be better understood in the context of Section \ref{sec:weak_type}. 
(See also \cite[Problem 7]{Arcozzi2010}, \cite{ARS2008, Tchoundja2008},
and Example \ref{exa:real_part_modulus}}).

We develop a potential theory for a class of reproducing kernel Hilbert spaces $\mathcal{H}$ on $\mathbb{B}_d$,
which includes in particular the Hardy space in the disc, the Dirichlet space and the Drury--Arveson space.
Detailed definitions will be given in Section \ref{sec:energy_basics}. Briefly,
if $K$ denotes the reproducing kernel of $\mathcal{H}$, we define the \emph{energy} of a 
Borel probability measure $\mu$ on $\partial \mathbb{B}_d$ to be
\begin{equation*}
  \mathcal{E}_{\mathcal{H}}(\mu) = \sup_{0 \le r < 1} \int_{\partial \mathbb{B}_d} \int_{\partial \mathbb{B}_d} \Re K(r z, rw) d \mu(w) d \mu(z) \in [0,\infty].
\end{equation*}
There is also a functional analytic interpretation of energy,
namely that $\mathcal{E}(\mu)$ is the square of the norm of the densely defined integration
functional $f \mapsto \int_{\partial \mathbb{B}_d} f \,d \mu$
on $\mathcal{H}$.
Having defined the energy, it is a standard procedure to define first
the \emph{$\mathcal H$-capacity} of a compact subset of $E \subset \partial \mathbb{B}_d$ as
\begin{equation*}
  \ccap_{\mathcal{H}}(E) = \sup \Big\{ \frac{1}{\mathcal{E}(\mu)}: \mu \text{ is a probability
  measure supported on } E \Big\},
\end{equation*}
and from there the \emph{inner capacity} and the \emph{outer capacity} $\ccap_{\mathcal{H}}^*(E)$ of an arbitrary subset $E \subset \partial \mathbb{B}_d$;
see Section \ref{sec:energy_basics} for details.
We will also show that the capacity can be equivalently expressed
in terms of a dual minimization problem; see Section \ref{sec:dual_formulation}.
In the Drury--Arveson space, our principal case of interest,
the capacity is a Choquet capacity, which implies in particular
that inner and outer capacity agree for all Borel sets;
see Theorem \ref{thm:Choquet} and the remark following it.

Let us consider some examples. In the Dirichlet space $\mathcal{D}$, we can take
\begin{equation*}
  K_\mathcal{D}(z,w) = \log \Big( \frac{e}{1 - z \overline{w}} \Big),
\end{equation*}
and a simple argument using Fatou's lemma and the dominated convergence theorem shows that in the formula for the energy,
we may take the limit $r \to 1$ inside the integral, hence
\begin{equation*}
  \mathcal{E}_{\mathcal{D}}(\mu) = \int_{\partial \mathbb{D}} \int_{\partial \mathbb{D}}
  \log \Big| \frac{e}{z- w} \Big| \, d \mu(w) d \mu(z).
\end{equation*}
This is (comparable to) the usual logarithmic energy, and so our capacity $\ccap_{\mathcal{D}}$ is (comparable to) the usual logarithmic
capacity \cite[Definition 2.4.1]{EKM+14}; see also Example \ref{exa:basic_capacities}.

On the other hand, in the Hardy space $H^2(\mathbb{D})$, we have $K_{H^2(\mathbb{D})}(z,w) = (1 - z \overline{w})^{-1}$,
and so expanding the kernel in a geometric series, we find that
\begin{equation*}
  \mathcal{E}_{H^2(\mathbb{D})}(\mu)
  = \sup_{0 \le r < 1} \sum_{n=0}^\infty r^{2 n} | \widehat{\mu}(n)|^2
  \approx \sum_{n=-\infty}^\infty |\widehat{\mu}(n)|^2.
\end{equation*}
Here, $\widehat{\mu}(n)$ denotes the $n$-th Fourier coefficient of $\mu$,
and $A \approx B$ means that there exists a constant $C \in (0,\infty)$
with $A \le C B$ and $B \le C A$.
It follows that $\mathcal{E}_{H^2(\mathbb{D})}(\mu) < \infty$ if and only if
  $\mu$ is absolutely continuous
  with respect to Lebesgue measure $m$ and the density $ f = \frac{d \mu}{d m}$ belongs to $L^2$.
  Moreover, $\mathcal{E}(\mu) \approx \|f\|^2_{L^2}$,
  so for a compact set $E \subset \partial \mathbb{D}$,
  \begin{equation*}
    \ccap_{H^2(\mathbb{D})}(E) \approx \sup \Big\{ \frac{1}{\|f\|_{L^2}^2}, \|f\|_{L^1} = 1, f \text{ supported on E} \Big\}.
  \end{equation*}
  The Cauchy--Schwarz inequality shows that if $f$ is supported on $E$,
  then $\|f\|_{L^1} \le m(E)^{1/2} \|f\|_{L^2}$, with equality if $f$ is the characteristic
  function of $E$. Thus, $\ccap_{H^2(\mathbb{D})}(E) \approx m(E)$,
  i.e.\ our capacity is comparable to Lebesgue measure.

  In the examples of Dirichlet space and Hardy space, the capacity $\ccap_{\mathcal{H}}$
  captures the usual notions of size of subset of $\partial \mathbb{D}$ that is adapted to the space $\mathcal{H}$.
  We establish results that show that this principle remains true in more general
  spaces.
  In the case of the Drury--Arveson space, we obtain the strongest results,
  so we focus on this case now.

Our first main result about $H^2_d$ can then be stated as follows.

\begin{theorem}
  \label{thm:boundary_limits_intro}
  For each $f \in H^2_d$, there exists a Borel set $E \subset \partial \mathbb{B}_d$
  of outer $H^2_d$-capacity zero such that $f$ has a \Kor\ limit at every point in $\partial \mathbb{B}_d \setminus E$.
\end{theorem}

This result will be proved in Section \ref{sec:weak_type}.
The key tool is a capacitary weak-type estimate for the maximal function.

Theorem \ref{thm:boundary_limits_intro} is sharp in the following sense.

\begin{theorem}
  \label{thm:unboundedness_intro}
  Let $E \subset \partial \mathbb{B}_d$ be a compact set of $H^2_d$-capacity zero.
  Then there exists $f \in H^2_d$ with $\lim_{r \nearrow 1} |f(r \zeta)| = \infty$ for all $\zeta \in E$.
\end{theorem}
This theorem will be proved in Section \ref{sec:unboundedness}.

Theorems \ref{thm:boundary_limits_intro} and \ref{thm:unboundedness_intro} suggest
that that the notion of capacity zero is the right notion of smallness in the context of boundary behavior
in the Drury--Arveson space.
There is another notion of smallness, called {totally null}, which proved to be very useful
in the context of functional calculi \cite{BHM18,CD16a}, ideals and peak interpolation
in Arveson's algebra $\mathcal A_d$
\cite{Clouatre2016,CD18,DH23}.
The definition is due to Clou\^atre and Davidson \cite[Definition 5.5]{Clouatre2016}; it is motivated by classical work on the ball algebra (see also \cite[Chapter 9]{Rudin08} for a detailed account of the classical theory).

Denoting the multiplier algebra of $H^2_d$ by $\Mult(H^2_d)$, one says that a complex Borel measure $\mu$ on the sphere
is \emph{$\Mult(H^2_d)$-Henkin} if for every sequence $(p_n)$ of polynomials
with $\sup_{n} \|p_n\|_{\Mult(H^2_d)} < \infty$ and $\lim_{n \to \infty} p_n(z) = 0$
for all $z \in \mathbb{B}_d$, we have
\begin{equation*}
  \lim_{n \to \infty} \int_{\partial \mathbb{B}_d} p_n \, d \mu = 0.
\end{equation*}
A Borel subset $E \subset \partial \mathbb{B}_d$ is called \emph{$\Mult(H^2_d)$-totally null} if $|\mu|(E)  = 0$
for all $\Mult(H^2_d)$-Henkin measures $\mu$.

In the case of the Dirichlet space and certain weighted Dirichlet spaces,
it was shown in \cite{CH20} that the appropriate versions of the totally null and the capacity
zero condition are equivalent.
This is also true for our $H^2_d$-capacity.

\begin{theorem}
\label{thm:cap_TN_intro}
  A Borel set $E \subset \partial \mathbb{B}_d$ has (outer) $H^2_d$-capacity zero if and only if it is $\Mult(H^2_d)$-totally null.
\end{theorem}

This result will be proved in Section \ref{sec:TN_Henkin}.

Henkin measures for the Hardy space on the ball are characterized by theorems
of Henkin and Cole--Range; see \cite[Chatper 9]{Rudin08}.
It is an open problem to characterize Henkin measures for the Drury--Arveson space;
see \cite[Section 5]{Clouatre2016} and \cite{Hartz17} for more discussion.
The next result gives one possible characterization.

\begin{theorem}
\label{thm:henkin_intro}
A complex Borel measure on $\partial \mathbb B_d$
is $\Mult(H^2_d)$-Henkin if and only if it is 
  absolutely continuous with respect
  to a probability measure of finite $H^2_d$-energy.
\end{theorem}

This result will also be proved in Section \ref{sec:TN_Henkin}.
We also obtain a characterization of  \emph{$\Mult(H^2_d)$-totally singular} measures, i.e.\ measures
that are singular with respect to every $\Mult(H^2_d)$-Henkin measure;
see Theorem \ref{thm:TS}.
These appear in the description
of the dual space of Arveson's algebra $\mathcal A_d$;
see \cite{Clouatre2016,DH23}.

In the Dirichlet space, capacity zero sets also play a crucial role
in the study of cyclic functions, as was shown by Brown and Shields \cite{BS84}.
We will prove a version of the Brown--Shields result for the Drury--Arveson space.
Recall that $f \in H^2_d$ is said to be \emph{cyclic} if there exists a sequence $(p_n)$ of
polynomials such that $(p_n  f)$ converges to $1$ in $H^2_d$.
We write $f^*(\zeta) = \lim_{r \nearrow 1} f(r \zeta)$, which exists
for all $\zeta \in \partial \mathbb B_d$ outside a set of outer capacity zero by Theorem \ref{thm:boundary_limits_intro}. We also let
\[
    Z(f^*) = \{\zeta \in \partial \mathbb B_d: f^*(\zeta) = 0 \}.
\]
Then our notion of capacity yields the following necessary condition
for cyclicity, which will be proved in Section \ref{sec:cyclic}.

\begin{theorem}
\label{thm:cyclic_intro}
    If $f \in H^2_d$ is cyclic, then $Z(f^*)$ has outer $H^2_d$-capacity zero.
\end{theorem}

Whereas our presentation focused on the Drury--Arveson space so far,
it is worth discussing for which other spaces our results hold.
We carry out the definition of capacity for unitarily invariant reproducing
kernel Hilbert spaces on the ball.
However, many of the basic properties of capacity such as subadditivity require
non-negativity of the real part of the reproducing kernel. For instance, this excludes the Bergman space and the Hardy space $H^2(\mathbb{B}_d)$ for $d\geq 2$.

The proof of Theorem \ref{thm:boundary_limits_intro}, the existence of \Kor\ limits,
hinges on the capacitary weak-type inequality for the maximal function,
which in turn depends on the very particular form of the Drury--Arveson kernel.
Our arguments also apply to standard weighted Dirichlet spaces
on the ball, namely the spaces $\mathcal D_a$ for $a \in [0,1)$,
whose reproducing kernel is
\[
    \frac{1}{(1 - \langle z,w \rangle)^a}
\]
in case $a \neq 0$ and
\[
    \log \Big( \frac{e}{1 - \langle z,w \rangle} \Big)
\]
in case $a = 0$.
In fact, for those spaces, the argument simplifies because the real part and the modulus
of the kernel are comparable. This leads to a new proof of the Hilbert
space case of a theorem of Ahern and Cohn \cite{AC89};
see Remark \ref{rem:AC_new_proof}.
Our proof that the capacity is a Choquet capacity also uses the weak-type inequality
for the maximal function, so it is specific to the Drury--Arveson space and standard
weighted Dirichlet spaces.

Even though the existence of \Kor\ limits depends on the particular form of the kernel,
we establish a result about radial limits that applies to any unitarily
invariant space whose kernel has non-negative real part.
The price we have to pay is that we have to settle for a weaker notion
of convergence, namely convergence in capacity. The precise statement
can be found in Theorem \ref{thm:convergence_in_capacity}.

Theorem \ref{thm:unboundedness_intro} holds provided
the reproducing kernel has non-negative real part;
see Theorem \ref{thm:cap_0_unboundedness}.
As for Theorem \ref{thm:cap_TN_intro}, the implication ``totally null
implies capacity zero'' is true in general (at least for compact sets), while the converse depends on the complete Pick property; see Theorem \ref{thm:TN_capacity}. Similarly in Theorem \ref{thm:henkin_intro} the ``if'' part holds true for general spaces while the ``only if'' part requires the complete Pick property.

The necessity of the capacitary condition for cyclicity
in Theorem \ref{thm:cyclic_intro} depends on the capacitary weak type inequality,
so the result is again specific to the Drury--Arveson space and standard
weighted Dirichlet spaces as stated.
But if one is willing to settle for the a priori weaker conclusion of
inner capacity zero, then the result holds for general spaces
whose kernel has non-negative real part; see Corollary \ref{cor:cyclic_inner_cap}.

\section{Energy and capacity in RKHS}
\label{sec:energy_basics}

When proving Fatou's theorem for $H^2(\mathbb D)$, it is convenient to also consider harmonic functions
in addition to holomorphic ones. In the case of the Drury--Arveson space $H^2_d$,
whose reproducing kernel is $K_{H^2_d}(z,w) = \frac{1}{1 - \langle z,w \rangle}$,
we consider a space of pluriharmonic functions. Explicitly, we let $\PHDA$
be the reproducing kernel Hilbert space on $\mathbb{B}_d$ with kernel
\[
    k_{h^2_d}(z,w) =
    2 \Re \frac{1}{1 - \langle z,w \rangle} - 1 =
    \frac{1 - |\langle z,w \rangle|^2}{|1 - \langle z,w \rangle|^2}.
\]
The representation
\begin{equation*}
    k_{h^2_d}(z,w) =
    \sum_{n=0}^\infty \langle z,w \rangle^n + \sum_{n=1}^\infty \overline{\langle z,w \rangle}^n
\end{equation*}
shows that $k_{h^2_d}$ is indeed a positive semi-definite function.
We call $\PHDA$ the \emph{pluriharmonic Drury--Arveson space}.
If $d=1$, then $k_{\PHDA}$ is the Poisson kernel of $\mathbb{D}$, and $\PHDA$ is the harmonic Hardy space $h^2(\mathbb{D})$.

In order to deal with the holomorphic and the pluriharmonic setting at the same time,
we make the following definition.

\begin{definition}
  \label{defn:pluriharmonic_space}
A \emph{unitarily invariant space of pluriharmonic functions} on $\mathbb{B}_d \subset \mathbb{C}^d$
is a reproducing kernel Hilbert space $\mathcal{H}$ of functions on $\mathbb{B}_d$ whose reproducing kernel is of the form
\begin{equation*}
  K(z,w) = \sum_{n=0}^\infty a_n \langle z,w \rangle^n + \sum_{n =1}^\infty a_{-n} \overline{ \langle z,w \rangle}^n,
\end{equation*}
where $a_n \ge 0$ for all $n \in \mathbb{Z}$, $a_0 =1$ and the series converges on $\mathbb{B}_d \times \mathbb{B}_d$.
\end{definition}

The most important examples for us occur when
\begin{enumerate}
  \item $a_n = 0$ for all $n < 0$, in which case $\mathcal{H}$ is a unitarily invariant space of holomorphic
    functions on $\mathbb{B}_d$, or
  \item $a_{-n} = a_n$ for all $n \in \mathbb{Z}$, in which case $K$ is real-valued.
\end{enumerate}
In the setting of Definition \ref{defn:pluriharmonic_space}, let us write $K = K_+ + \overline{K_-}$, where
\begin{equation*}
  K_+(z,w) = \sum_{n=0}^\infty a_n \langle z,w \rangle^n \quad \text{ and } \quad
  K_-(z,w) = \sum_{n=1}^\infty a_{-n} \langle z,w \rangle^n.
\end{equation*}
Then $K_+$ and $K_-$ are reproducing kernels of spaces of holomorphic functions,
say of $\mathcal{H}_+$ and $\mathcal{H}_-$, respectively.
We collect a few basic and straightforward properties of our spaces.

\begin{lemma}
  \label{lem:pluriharmonic_RKHS_basic}
  With notation as above, the following statements hold:
  \begin{enumerate}[label=\normalfont{(\alph*)}]
    \item Every function $f \in \mathcal{H}$ has a unique decomposition
      \begin{equation*}
        f = g + \overline{h},
      \end{equation*}
      where $g \in \mathcal{H}_+, h \in \mathcal{H}_-$. Conversely, any function of this form belongs to $\mathcal{H}$.
      Moreover,
      \begin{equation*}
        \|f\|_{\mathcal{H}}^2 = \|g\|_{\mathcal{H}_+}^2 + \|h\|_{\mathcal{H}_-}^2.
      \end{equation*}
    \item The set
      \begin{equation*}
        \{ z^\alpha: \alpha \in \mathbb{N}_0^d, a_{|\alpha|} \neq 0 \}
        \cup \{ \overline{z}^\beta: \beta \in \mathbb{N}_0^d \setminus \{0\}, a_{-|\beta|} \neq 0 \}
      \end{equation*}
      forms an orthogonal basis of $\mathcal{H}$. Moreover,
      \begin{equation*}
        \|z^\alpha\|^2 = \frac{\alpha!}{|\alpha|!} \frac{1}{a_{|\alpha|}}
          \quad \text{ and } \quad
        \|\overline{z}^\beta\|^2 = \frac{\beta!}{|\beta|!} \frac{1}{a_{-|\beta|}} .
      \end{equation*}
    \item
  Let $f: \mathbb{B}_d \to \mathbb{C}$ be a function.
  Then $f \in \mathcal{H}$ if and only if $f_r \in \mathcal{H}$
  for all $r \in [0,1)$ and $\sup_{0 \le r < 1} \|f_r\| < \infty$.
  In this case, $\|f_r\|$ is an increasing function of $r$,
  $\|f\| = \sup_{0 \le r < 1} \|f_r\|$
  and $\lim_{r \to 1} \|f - f_r\| = 0$.
  \end{enumerate}
\end{lemma}

\begin{proof}
  (a) If $g_1,g_2 \in \mathcal{H}_+$ and $h_1,h_2 \in \mathcal{H}_-$ satisfy
  \begin{equation*}
    g_1 + \overline{h_1} = g_2 + \overline{h_2},
  \end{equation*}
  then
  \begin{equation*}
    g_1 - g_2 = \overline{h_1 - h_2}.
  \end{equation*}
  So $h_1 - h_2$ and $\overline{h_1 - h_2}$ are both holomorphic,
  hence $h_1 - h_2$ is constant. Moreover, $h_1(0) = h_2(0) = 0$, as the functions in $\mathcal{H}_-$ vanish
  at $0$. So $h_1 = h_2$.
  This shows uniqueness of the decomposition.

  Next, we define
  \begin{equation*}
    \mathcal{K} = \{ g + \overline{h}, g \in \mathcal{H}_+, h \in \mathcal{H}_-\}
  \end{equation*}
  with inner product
  \begin{equation*}
    \langle g_1 + \overline{h_1}, g_2 + \overline{h_2} \rangle_{\mathcal{K}}
    = \langle g_1,g_2 \rangle_{\mathcal{H}_+}
    + \overline{\langle h_1, h_2 \rangle}_{\mathcal{H}_-} .
  \end{equation*}
  If $w \in \mathbb{B}_d$, then $K(\cdot,w) = K_+(\cdot,w) + \overline{K_-(\cdot,w)} \in \mathcal{K}$, and
  \begin{equation*}
    \langle g + \overline{h}, K(\cdot,w) \rangle_{\mathcal{K}}
    = \langle g, K_+(\cdot,w) \rangle_{\mathcal{H}_+}
    + \overline{\langle h, K_-(\cdot,w) \rangle}_{\mathcal{H}_-} 
    = g(w) + \overline{h(w)}.
  \end{equation*}
  Thus, $\mathcal{K}$ is a reproducing kernel Hilbert space on $\mathbb{B}_d$,
  whose reproducing kernel is $K$.
  It follows that
  $\mathcal{H} = \mathcal{K}$ with equal norms.

  (b) Follows from (a) and a standard result about unitarily invariant spaces of holomorphic functions.
  A proof can for instance be found in \cite[Proposition 4.1]{GHX04}.

  (c) Let $f \in \mathcal{H}$. By (b), we may expand
  \begin{equation*}
    f(z) = \sum_{\alpha \in \mathbb{N}_0^d} \widehat{f}(\alpha) z^\alpha
    + \sum_{\beta \in \mathbb{N}_0^d \setminus \{0\}} \widecheck{f}(\beta) \overline{z}^\beta
  \end{equation*}
  for suitable scalars $\widehat{f}(\alpha), \widecheck{f}(\beta) \in \mathbb{C}$.
  Then
  \begin{equation*}
    f_r(z) = \sum_{\alpha \in \mathbb{N}_0^d} \widehat{f}(\alpha) r^{|\alpha|} z^\alpha
    + \sum_{\beta \in \mathbb{N}_0^d \setminus \{0\}} \widecheck{f}(\beta) r^{|\beta|}
    \overline{z}^\beta,
  \end{equation*}
  so another application of (b) shows that $f_r \in \mathcal{H}$, that $\|f_r\|_{\mathcal{H}}$ is increasing
  as a function of $r$ and that $\|f\|_\mathcal{H} = \sup_{0 \le r < 1} \|f_r\|_\mathcal{H}$.
  An elementary estimate or the dominated convergence theorem for sums shows that $\lim_{r \to 1} \|f - f_r\|_\mathcal{H} = 0$.

  Conversely, let $f_r \in \mathcal{H}$ for all $r \in [0,1)$ and $\|f_r\|_\mathcal{H}$ be bounded.
  Since the linear span of the kernel functions is dense in $\mathcal{H}$, it follows
  that $(f_r)$ converges weakly to a function in $\mathcal{H}$ as $r \to 1$;
  the weak limit necessarily equals $f$, so $f \in \mathcal{H}$.
\end{proof}

Let $\mu \in M(\overline{\mathbb{B}_d})$ be a complex Borel measure
on $\overline{\mathbb{B}_d}$. 
We will later restrict attention to measures supported on the boundary $\partial \mathbb{B}_d$,
but for now, it is instructive to work in this more general setting. We recall that such measures are automatically inner and outer regular. 
For $0 \le r < 1$, we define
\begin{equation*}
  \mathcal{E}_r(\mu) = \int_{\overline{\mathbb B_d}} \int_{\overline{\mathbb B_d}} K(r z, r w) d \mu(w) d \overline{\mu}(z).
\end{equation*}
For $\alpha \in \mathbb{N}_0^d$, let
\begin{equation*}
  \widehat{\mu}(\alpha) = \int_{\overline{\mathbb B_d}} \overline{z^\alpha} \, d \mu(z)
  \quad \text{ and } \quad
  \widecheck{\mu}(\alpha) = \int_{\overline{\mathbb B_d}} z^\alpha \, d \mu(z).
\end{equation*}
The series defining $K$ converges uniformly on compact subsets of $\mathbb{B}_d$, so a small computation shows that
\begin{equation}
  \label{eqn:energy_formula}
  \mathcal{E}_r(\mu) =
  \sum_{n=0}^\infty a_n r^{2n} \sum_{|\alpha| = n} \frac{|\alpha|!}{\alpha!} |\widehat{\mu}(\alpha)|^2
  + \sum_{n=1}^\infty a_{-n} r^{2n} \sum_{|\alpha| = n} \frac{|\alpha|!}{\alpha !} |\widecheck{\mu}(\alpha)|^2.
\end{equation}
This shows that $\mathcal{E}_r(\mu)$ is a non-negative real number and increasing
as a function of $r$.

\begin{definition}
  \label{defn:energy}
  Let $\mathcal H$ be a unitarily invariant space of pluriharmonic functions
  on $\mathbb B_d$ with reproducing kernel $K$.
  Let $\mu \in M(\overline{\mathbb B_d})$ be a complex Borel measure
  on $\overline{\mathbb B_d}$. We define the \emph{$\mathcal{H}$-energy} of $\mu$ to be
  \begin{equation*}
    \mathcal{E}(\mu) = \mathcal{E}_{\mathcal H}(\mu) =  \sup_{0 \le r < 1} \mathcal{E}_r(\mu) = \sup_{0 \le r < 1} \int_{\overline{\mathbb B_d}} \int_{\overline{\mathbb B_d}} K(r z, r w) d \mu(w) d \overline{\mu}(z) \in [0,\infty].
  \end{equation*}
  If $E \subset \overline{\mathbb B_d}$ is compact, we
  let $P(E) \subset M(\overline{\mathbb B_d})$ be the set of Borel probability measures supported on $E$
  and
  define the \emph{$\mathcal{H}$-capacity} of $E$ to be
  \begin{equation*}
    \ccap(E) = \ccap_{\mathcal{H}}(E) = \sup \Big\{ \frac{1}{\mathcal{E}(\mu)} : \mu \in P(E) \Big\}.
  \end{equation*}
\end{definition}

Gramians of reproducing kernels are a key tool in the study of reproducing kernel Hilbert spaces.
The following example shows how to obtain them as energies of suitable measures.
\begin{example}
  If $z_1,\ldots,z_n \in \mathbb{B}_d$, $\lambda_1,\ldots,\lambda_n \in \mathbb{C}$ and $\mu = \sum_{j=1}^n \lambda_j \delta_{z_j}$
  is a finite linear combination of Dirac measures, then
  \begin{equation*}
    \mathcal{E}(\mu) = \lim_{r \to 1} \mathcal{E}_r(\mu)
    = \sum_{i=1}^n \sum_{j=1}^n K(z_i,z_j) \lambda_j \overline{\lambda_i}.
  \end{equation*}
  If $G$ is the Gramian of the reproducing kernels at the points $z_1,\ldots,z_n$
  and $\lambda = (\lambda_1,\ldots,\lambda_n) \in \mathbb{C}^n$, then
  \begin{equation*}
    \mathcal{E}(\mu) = \langle G \lambda , \lambda \rangle.
  \end{equation*}
\end{example}

We can extend the definition of capacity to all subsets of $\overline{\mathbb{B}_d}$ by approximating from within by compact sets.
That is, if $E\subset \overline{\mathbb{B}_d}$, we define
\[  \ccap(E)=\sup\{  \ccap(F): F\subset E, F \,\, \text{compact}\, \}. \]
We shall refer to this capacity as \emph{inner capacity}. 

Similarly, an \emph{outer capacity} can be defined by approximating from outside by open sets:
\[ \ccap^*(E) = \inf\{ \ccap(G): G \supset E, G \,\, \text{open} \, \}. \]
We refer to Appendix \ref{sec:app_abstract_capacities} for more details on this construction.

\begin{remark}
  \label{rem:energy_positive_kernel}
  The most important case for our purposes occurs when $\mu$ is a positive measure
  and $\Re K \ge 0$.
  In this case,
  \begin{equation*}
    \mathcal{E}(\mu) = \lim_{r \to 1} \int_{\overline{\mathbb B_d}} \int_{\overline{\mathbb B_d}}
    \Re K(r z, r w) d \mu(w) d \mu(z),
  \end{equation*}
  so the $\mathcal{H}$-energy can be understood as a limit of classical energy
  integrals with respect to a non-negative kernel.
  This also shows that Definition \ref{defn:energy} is consistent with the definition given in the introduction.
\end{remark}

We record the following basic observation for future reference.

\begin{lemma}
  \label{lem:capacity_zero}
  Let $\mathcal{H}$ be a unitarily invariant space of pluriharmonic functions
  whose kernel has non-negative real part.
  Let $E \subset \overline{\mathbb{B}_d}$ be Borel. Then $\ccap(E) = 0$
  if and only if $\mu(E) = 0$ for all $\mu \in P(\overline{\mathbb{B}_d})$ with $\mathcal{E}(\mu) < \infty$.
\end{lemma}

\begin{proof}
  Sufficiency is immediate from the definition.
  As for necessity, suppose that there exists a probability measure $\mu$
  with $\mathcal{E}(\mu) < \infty$ and $\mu(E) > 0$.
  By regularity of $\mu$, there exists a compact set $F \subset E$ with $\mu(F) > 0$.
  Since the real part of the kernel of $\mathcal{H}$
  is non-negative, it follows that the measure $\chi_F d \mu$ also has finite energy;
  see Remark \ref{rem:energy_positive_kernel}. Normalizing this measure,
  we see that $F$ has positive capacity, hence so does $E$.
\end{proof}

\begin{remark}
  Let $\mathcal{H}$ be a unitarily invariant space of holomorphic functions on $\mathbb{B}_d$
  with kernel $K$.
  Instead of working with $\Re K$, it is also natural to work with $k = 2 \Re K - 1$.
  The corresponding reproducing kernel Hilbert space $\widetilde{\mathcal H}$
  is a unitarily invariant space of pluriharmonic functions,
  whose kernel is real valued. 
  The space $\mathcal H$ embeds isometrically into $\widetilde{\mathcal H}$.
  For instance, if $\mathcal{H} = H^2(\mathbb{D})$, then $k$ is the Poisson kernel and $\widetilde{\mathcal{H}} = h^2(\mathbb{D})$, the harmonic Hardy space. 
  Working with $h^2(\mathbb{D})$ also has the advantage that the capacity
  turns out to be equal to Lebesgue measure, not just comparable to it,
  see Example \ref{exa:basic_capacities} (c) below.
  In general, if $\mu$ is a positive measure, then the energy formula \eqref{eqn:energy_formula} implies
  that $\mathcal{E}_{\mathcal{H}}(\mu)$ and $\mathcal{E}_{\widetilde{\mathcal{H}}}(\mu)$ are comparable.

  On the other hand, if we work with $\Re K$ instead of $2 \Re K - 1$,
  then the inclusion of the holomorphic space into the pluriharmonic space
  is bounded and bounded below. The energies of positive measures in the holomorphic
  and the pluriharmonic space agree.
\end{remark}

\begin{example}
  \label{exa:basic_capacities}
  (a) Let $\mathcal{H} = \mathcal{D}$, the classical Dirichlet space on the disc,
  with reproducing kernel
  \begin{equation*}
    K_\mathcal{D}(z,w) = \log \Big( \frac{e}{1 - z \overline{w}} \Big).
  \end{equation*}
  Then the energy formula \eqref{eqn:energy_formula}
  shows that for every positive measure $\mu \in M(\partial \mathbb{D})$,
  we have
  \begin{equation*}
    \mathcal{E}_{\mathcal{D}}(\mu) = |\widehat{\mu}(0)|^2 + \sum_{n=1}^\infty \frac{| \widehat{\mu}(n)|^2}{n},
  \end{equation*}
  which is comparable to the usual logarithmic energy;
  see for instance \cite[Theorem 2.4.4]{EKM+14}.
  In fact, we can see this directly, as in the proof of \cite[Theorem 2.4.4]{EKM+14},
  since
  \begin{align*}
    \int_{\partial \mathbb{D}} \int_{\partial \mathbb{D}} \log
    \Big| \frac{e}{1 - z \overline{w}} \Big| \, d \mu(w) d \mu(z)
    &= \lim_{r \to 1}
    \int_{\partial \mathbb{D}} \int_{\partial \mathbb{D}} \log
    \Big| \frac{e}{1 - r^2 z \overline{w}} \Big| \, d \mu(w) d \mu(z) \\
    &= \mathcal{E}_{\mathcal{D}}(\mu).
  \end{align*}
  Here, we used the dominated convergence theorem and the
  basic inequality $|\frac{1}{1 - r^2 z \overline{w}}| \le | \frac{2}{1 - z \overline{w}}|$
  if the double integral on the left is finite, and Fatou's lemma if the double integral is infinite.
  Therefore, $\ccap_{\mathcal{D}}$ is comparable to the usual logarithmic capacity. 

  (b) Let $\mathcal{H} = \mathcal{D}_a$ be the weighted Dirichlet space with reproducing kernel
  \begin{equation*}
    K_{\mathcal{D}_a}(z,w) = \frac{1}{(1 - z \overline{w})^a}
  \end{equation*}
  for $0 < a <1$. Then $\Re K_{\mathcal{D}_a}$ and $|K_{\mathcal{D}_a}|$ are comparable, and a similar argument as in part (a) shows that
  \begin{equation*}
    \mathcal{E}_{\mathcal{D}_a}(\mu) \approx
    \int_{\partial \mathbb{D}} \int_{\partial \mathbb{D}} \frac{1}{|1 - z \overline{w}|^a} \, d \mu(w) d \mu(z),
  \end{equation*}
  so $\ccap_{\mathcal{D}_a}$ is comparable to the Riesz capacity of degree $a$.
  
  (c)
  We already saw in the introduction that $\ccap_{H^2(\mathbb{D})}$ is comparable
  to Lebesgue measure.
  If we replace $H^2(\mathbb{D})$ with the harmonic Hardy space $h^2(\mathbb{D})$, whose reproducing kernel
  is the Poisson kernel
  \begin{equation*}
    K_{h^2(\mathbb{D})}(z,w) = P(z,w) =  \frac{1 - |z \overline{w}|^2}{|1 - z \overline{w}|^2},
  \end{equation*}
  then the same computation as in the introduction shows the exact formula $\ccap_{h^2(\mathbb{D})}(E) = m(E)$
  for all compact subsets $E \subset \partial \mathbb{D}$.  Moreover, $d \mu = \frac{\chi_E}{m(E)} d m$ is a probability measure supported on $E$ that minimizes
  the energy, i.e.\ an equilibrium measure.
  We also see that unlike the case of the (weighted) Dirichlet space, we cannot express
  $\mathcal{E}_{h^2(\mathbb{D})}(\mu)$ by taking the limit $r \to 1$ inside the integral,
  as the Poisson kernel $P$ vanishes for $z,w \in \partial \mathbb{D}$ with $z \neq w$.
\end{example}

We require the following functional analytic interpretation of energy.
The proof follows along the same lines as the proof of \cite[Proposition 3.2]{CH20}.
Since the present setting is slightly different and since some parts of the earlier proof
do not generalize, we provide the details.

\begin{proposition}
  \label{prop:energy_fa}
  Let $\mathcal H$ be a unitarily invariant space of pluriharmonic
  functions on $\mathbb B_d$ with reproducing kernel $K$.
  Let $\mu \in M(\overline{\mathbb B_d})$. Define
  \[
  f_\mu(z) = \int_{\overline{\mathbb B_d}} K(z,w) d \mu(w).
  \]
  The following are equivalent:
  \begin{enumerate}[label=\normalfont{(\roman*)}]
    \item $\mathcal{E}(\mu) < \infty$;
    \item  the densely defined integration functional $\rho_{\overline{\mu}}$ with respect to $\overline{\mu}$ is bounded on $\mathcal{H}$, i.e.\
  there exists a constant $C < \infty$ such that
  \begin{equation*}
    \Big| \int_{\overline{\mathbb B_d}} g \, d \overline{\mu} \Big| \le C \|g\|_{\mathcal{H}}
    \quad \text{ for all } g \in \mathcal{H} \cap C(\overline{\mathbb{B}_d}).
  \end{equation*}
    \item $f_\mu \in \mathcal{H}$.
  \end{enumerate}
  In this case, we have
  \begin{equation*}
    \|\rho_{\overline{\mu}}\|^2_{\mathcal{H}^*} = \|f_\mu\|^2 = \mathcal{E}(\mu)
  \end{equation*}
  and
  \begin{equation*}
    \rho_{\overline{\mu}}(g) = \langle g, f_\mu \rangle \quad \text{ for all } g \in \mathcal{H}.
  \end{equation*}
\end{proposition}

Notice that if $\mu = \delta_w$ is a Dirac measure at some $w \in \mathbb{B}_d$, then $\rho_{\overline{\mu}}$
is the functional of evaluation at $w$ and $f_\mu$ is the reproducing kernel at $w$.

\begin{proof}
  We argue as in \cite[Proposition 3.2]{CH20}.
  Let $f = f_\mu$ and $k_w(z) = K(z,w)$.
  Let $0 \le r < 1$.
  Since the function
  \begin{equation*}
    \overline{\mathbb B_d} \to \mathcal{H}, \quad w \mapsto k_{r w},
  \end{equation*}
  is continuous, 
  the integral
  \begin{equation*}
    \int_{\overline{\mathbb B_d}} k_{r w} d \mu(w)
  \end{equation*}
  converges in $\mathcal{H}$, and its value at $z$ equals $f_\mu(r z)$, hence $f_r \in \mathcal{H}$.
  Moreover, if $g \in \mathcal{H}$, then
  \begin{equation}
    \label{eqn:energy_fa}
    \langle g,f_r \rangle = \int_{\overline{\mathbb B_d}} \langle g, k_{r w} \rangle \, d \overline{\mu}(w)
    = \int_{\overline{\mathbb B_d}} g_r \, d \overline{\mu}.
  \end{equation}

  (iii) $\Rightarrow$ (ii) Assume $f \in \mathcal{H}$. Then by Lemma \ref{lem:pluriharmonic_RKHS_basic},
  we have $\lim_{r \to 1} f_r = f$ in $\mathcal{H}$. Thus,
  for each $g \in \mathcal{H} \cap C( \overline{\mathbb{B}_d})$, we may take the limit $r \to 1$
  in Equation \eqref{eqn:energy_fa} to conclude that
  $\langle g,f \rangle = \int_{\overline{\mathbb B_d}} g \, d \overline{\mu}$.
  Hence (ii) holds by the Cauchy--Schwarz inequality. This argument also establishes
  the final statement of the proposition and the equality $\|\rho_{\overline{\mu}}\|_{\mathcal{H}^*} = \|f_\mu\|$.

  (ii) $\Rightarrow$ (iii)  Choosing $g = f_r$ in Equation \eqref{eqn:energy_fa}
  and using Lemma \ref{lem:pluriharmonic_RKHS_basic} (c), we find that
  \begin{equation*}
    \|f_r\|^2 \le C \|f_{r^2}\| \le C \|f_r\|,
  \end{equation*}
  so $\|f_r\| \le C$ for all $r \in [0,1)$, hence $f \in \mathcal{H}$ by Lemma \ref{lem:pluriharmonic_RKHS_basic} (c).

  (i) $\Leftrightarrow$ (iii)
  Choosing $g = f_r$ in Equation \eqref{eqn:energy_fa}, we find that
  \begin{equation*}
    \|f_r\|^2 = \int_{\overline{\mathbb B_d}} f(r^2 z) d \overline{\mu}(z)
    = \int_{\overline{\mathbb B_d}} \int_{\overline{\mathbb B_d}} K(r z, r w)
    d \mu(w)
      d \overline{\mu}(z) = \mathcal{E}_r(\mu).
  \end{equation*}
  Taking the supremum over $r$ and recalling Lemma \ref{lem:pluriharmonic_RKHS_basic} (c) yields
  the equivalence of (i) and (iii) and the the equality $\mathcal{E}(\mu) = \|f\|^2$.
\end{proof}

We now consider some examples in the Drury--Arveson space.
In particular, we see that our capacity is different
from that developed in \cite{AC89} using the modulus of the kernel.
\begin{example}
  \label{exa:real_part_modulus}
  \begin{enumerate}[label=\normalfont{(\alph*)},wide]
    \item 
    A measure $\mu \in P(\partial \mathbb B_d)$ is called a \emph{representing measure
    for $0$} if
    \[
        \int_{\partial \mathbb B_d} f \, d \mu = f(0)
    \]
    for all continuous functions $f$ on $\overline{\mathbb B_d}$ that are holomorphic
    on $\mathbb B_d$. It follows from Proposition \ref{prop:energy_fa} that every
    representing measure for $0$ has finite energy for $H^2_d$.
    For example, let $\mu$ be the one-dimensional Lebesgue measure on the circle
    \[
    C = \{(z,0):z \in   \partial \mathbb{D}\}.
    \]
    Formally, $\mu$ is the pushforward of normalized Lebesgue measure $m$ on $\partial \mathbb{D}$ by
    the map $\iota: z \mapsto (z,0)$.
    Then $\mu$ is a representing measure for $0$ and hence has finite energy.
    In particular, $\ccap_{H^2_d}(C) > 0$.
  \item Example \ref{exa:basic_capacities} (c) shows that the capacity
    of a subset of the circle $C$ of part (a) is essentially given
    by the one dimensional Lebesgue measure.
  \item
    If one defines energy and capacity using $|K_{H^2_d}|$
    (instead of $K_{H^2_d}$ itself or $\Re K_{H^2_d}$), then one finds
    for the measure $\mu$ of part (a) that
    \begin{align*}
       \int_{\partial \mathbb B_d} \int_{\partial \mathbb B_d}
        |K_{H^2_d}(r z, r w)| \, d \mu(w) d \mu(z)
        &=  \int_{0}^{2 \pi} \int_{0}^{2 \pi}
        \frac{1}{|1 - r^2 e^{i (t - s)}|} \frac{ds}{2 \pi}  \frac{ dt}{2 \pi} \\
        &= \int_{0}^{2 \pi} \frac{1}{|1 - r^2 e^{ i t }|} \frac{dt}{2 \pi}
    \end{align*}
    for $0 \le r < 1$.
    By Fatou's lemma, it follows that
    \begin{equation*}
      \sup_{0 \le r < 1} \int_{\partial \mathbb B_d} \int_{\partial \mathbb B_d}
        |K_{H^2_d}(r z, r w)| \, d \mu(w) d \mu(z)
      \ge \int_{0}^{2 \pi} \frac{1}{|1 - e^{it}|} \frac{dt}{2 \pi} = \infty.
    \end{equation*}
    This shows that $\mu$ has infinite energy with respect to $|K_{H^2_d}|$.

  \item
    Expanding on (c), let $\nu \in P(C)$. Then $\nu = \iota_* \tau$ for
    some $\tau \in P(\partial \mathbb{D})$, and if $0 \le r < 1$, then
    \begin{equation*}
       \int_{\partial \mathbb B_d} \int_{\partial \mathbb B_d}
       |K_{H^2_d}(rz , rw)| \, d \nu(w) d \nu(z)
       = \int_{\partial \mathbb{D}} \int_{\partial \mathbb{D}}
       \frac{1}{|1 - r^2 \xi \overline{\eta}|} d \tau(\eta) d \tau(\xi).
    \end{equation*}
    Let $h: \partial \mathbb{D} \to \mathbb{C}, h(\xi) = |1 - r^2 \xi|^{-1}$.
    The representation 
    \begin{equation*}
      h(\xi) = (1 - r^2 \xi)^{-1/2} \overline{(1 - r^2 \xi)^{-1/2}}
    \end{equation*}
    shows that $h$ has non-negative Fourier coefficients $\widehat{h}(n)$,
    as $\xi \mapsto (1 - r^2 \xi)^{-1/2}$ has non-negative Fourier coefficients.
    Therefore, expanding the integrand above in a series and using that
    $\tau$ is a probability measure,
    we find that
    \begin{equation}
      \label{eqn:ineq_energy}
      \begin{split}
       \int_{\partial \mathbb B_d} \int_{\partial \mathbb B_d}
       |K_{H^2_d}(rz , rw)| \, d \nu(w) d \nu(z)
       &= \sum_{n \in \mathbb{Z}} \widehat{h}(n) |\widehat{\tau}(n)|^2 \\
       &\ge \widehat{h}(0) \\
       &= \int_{0}^{2 \pi} \frac{1}{|1 - r^2 e^{ it}|} \, dt.
      \end{split}
    \end{equation}
    As in (c), we conclude that
    \begin{equation*}
      \sup_{0 \le r < 1}
       \int_{\partial \mathbb B_d} \int_{\partial \mathbb B_d}
       |K_{H^2_d}(rz , rw)| \, d \nu(w) d \nu(z) = \infty.
    \end{equation*}
    Thus, $C$ has capacity zero with respect to $|K_{H^2_d}|$.
    (Inequality \eqref{eqn:ineq_energy} can also be understood
    in the context of rotionally invariant reproducing kernel
    Hilbert spaces: non-negativity of the Fourier coefficients of $h$
    corresponds to positive semi-definiteness of the kernel $|K_{H^2_d}|$;
    the right-hand side of \eqref{eqn:ineq_energy} yields the energy
    of the measure $\mu$ with respect to the subspace of homogeneous functions of degree $0$;
    see for instance \cite[Lemma 6.2]{Hartz17a}).
  \item
    Let $d = 2$.
    We consider a class of subsets of $\partial \mathbb{B}_2$ that
    was already studied by Ahern and Cohn; see \cite[Theorem 2.11]{AC89}.
    and also \cite[Example 6.6]{Clouatre2016}.
    Let $E \subset \partial \mathbb{D}$ be compact, let
    \begin{equation*}
    h: \partial \mathbb{D} \times E \to \partial \mathbb{B}_2, \quad (\zeta_1,\zeta_2) \mapsto 2^{-1/2} (\zeta_1, \overline{\zeta_1} \zeta_2)
    \end{equation*}
    and let $F = h(\partial \mathbb{D} \times E) \subset \partial \mathbb{B}_2$.
    We claim that 
    \begin{equation}
      \label{eqn:cap_DA_Dirichlet}
      \ccap_{H^2_2}(F) = \ccap_{\mathcal{D}_{1/2}}(E),
    \end{equation}
    where $\mathcal{D}_{1/2}$ the the weighted Dirichlet
    space mentioned in Example \ref{exa:basic_capacities} (b).
    In particular, taking $E = \{1\}$, we find that
    the circle
    \begin{equation*}
      \{ 2^{-1/2} (z, \overline{z}): z \in \partial \mathbb{D} \} \subset \partial \mathbb{B}_2
    \end{equation*}
    has $H^2_2$-capacity zero, since the Dirac measure $\delta_1$ has infinite energy for $\mathcal{D}_{1/2}$.
    Thus, Examples (b) and (e) taken together show that our capacity
    is sensitive to the structure of $\partial{\mathbb{B}_d}$ as a CR manifold.

    To prove \eqref{eqn:cap_DA_Dirichlet},
    we use some known relationships between $\mathcal{D}_{1/2}$ and $H^2_2$,
    which for instance can be found in \cite{Hartz17}.
    Let $r(z_1,z_2) =2 z_1 z_2$. Then $r(F) = E$,
    so the push-forward $r_* \mu$ is a probability measure
    supported on $E$. If $g \in \mathcal{D}_{1/2}$,
    then $g \circ r \in H^2_2$ with $\|g \circ r\|_{H^2_2}
    = \|g\|_{\mathcal{D}_{1/2}}$ (for a proof,
    see for instance \cite[Lemma 4.1]{Hartz17}).
    Thus, assuming that $g$ is additionally continuous on $\overline{\mathbb{D}}$,
    it follows from Proposition \ref{prop:energy_fa} that
    \begin{equation*}
      \Big| \int_E g \, d (r_* \mu) \Big|
      = \Big| \int_F (g \circ r) \,d \mu \Big|
      \le \mathcal{E}_{H^2_2}^{1/2} (\mu) \|g \circ r \|_{H^2_2}
      = \mathcal{E}_{H^2_2}^{1/2} (\mu) \|g \|_{\mathcal{D}_{1/2}}.
    \end{equation*}
    Again by Proposition \ref{prop:energy_fa}, we find that $\mathcal{E}_{\mathcal{D}_{1/2}} (r_* \mu)
    \le \mathcal{E}_{H^2_2}(\mu)$.
    By definition of the capacities, we conclude that $\ccap_{H^2_2}(F) \le \ccap_{\mathcal{D}_{1/2}}(E)$.

    Conversely, let $\sigma$ be a $\mathcal{D}_{1/2}$-equilibrium measure for $E$
    and let $g \in \mathcal{D}_{1/2}$ be the holomorphic potential of $\sigma$.
    Let $\mu = h_*(m \otimes \sigma)$, where $m$ is the normalized
    Lebesgue measure on $\partial \mathbb{D}$.
    The proof of \cite[Lemma 4.6]{Hartz17} and Proposition \ref{prop:energy_fa} show that
    \begin{equation*}
      \int_{\partial \mathbb{B}_2} p \, d \mu = \langle p, g \circ r \rangle_{H^2_2}
    \end{equation*}
    for all polynomials $p$. Hence
    \begin{equation*}
      \mathcal{E}_{H^2_2}(\mu) = \|g \circ r\|_{H^2_2}^2 = \|g\|_{\mathcal{D}_{1/2}}^2
      = \mathcal{E}_{\mathcal{D}_{1/2}}(\sigma) = \ccap_{\mathcal{D}_{1/2}}(E)^{-1}.
    \end{equation*}
    This shows the remaining inequality in \eqref{eqn:cap_DA_Dirichlet}.
  \end{enumerate}
\end{example}

\section{Unboudnedness sets}
\label{sec:unboundedness}

The goal of this section is to prove Theorem \ref{thm:unboundedness_intro}.
We will make use of the abstract theory of capacity that we develop in the appendix;
it allows for a uniform treatment.

For the rest of the section, let $\mathcal H$ be a unitarily invariant space
of plurihamornic functions whose reproducing kernel $K$ has non-negative real part.
We write $k = \Re K$.
For $0 \le r < 1$, we introduce the mixed $r$-energy of two positive measures to be 
\[ \mathcal{E}_r(\mu,\nu)=\int_{\overline{\mathbb{B}}_d}\int_{\overline{\mathbb{B}}_d} k(rz,rw) d\mu(w) d\nu(z). \]
We also write $\mathcal E_r(\mu) = \mathcal{E}_r(\mu,\mu)$.
If $\mathcal{E}(\mu) < \infty$ and $\mathcal{E}(\nu)<\infty$,
then Proposition \ref{prop:energy_fa} shows that
\[
    \mathcal E_r(\mu,\nu)
    = \Re \int_{\overline{\mathbb B_d}} (f_\mu)_{r^2} \, d \nu
    = \Re \langle (f_{\mu})_{r^2}, f_{\nu} \rangle,
\]
so the limit
\[
    \mathcal E(\mu,\nu) =
    \lim_{r\to 1}\mathcal{E}_r(\mu,\nu)
    = \Re
    \langle f_\mu, f_{\nu} \rangle
\]
exists by Lemma \ref{lem:pluriharmonic_RKHS_basic} (c).
Observe that $\mathcal E(\mu,\mu) = \mathcal E(\mu)$, so the new definition
is consistent with the old one.

\begin{lemma}
  Let $0\le r<1$. In the sense of Definition \ref{defn:energy_functional} the energy $\mathcal{E}_r$ is an energy functional on the cone $M_+(\overline{\mathbb B_d})$, while $\mathcal{E}$ is an energy functional on the convex cone of measures satisfying $\mathcal{E}(\mu)<\infty$.
\end{lemma}

\begin{proof}
    The characterization in part (ii) of Proposition \ref{prop:energy_fa} shows that
    \[
        C = \{\mu \in M_+(\mathbb B_d): \mathcal E(\mu) < \infty \}
    \]
    is indeed a convex cone. It is also clear that $C$ is closed
    under restrictions to compact subsets of $\overline{\mathbb B_d}$.
    
    Properties $(A1)$ and $(A2)$ are clear for both energies. Notice that the kernel $k(r\cdot,r\cdot)$ is continuous on $\overline{\mathbb{B}_d} \times \overline{\mathbb{B}_d}$. If $\mu_n$ converges weak* to $\mu$ then the product measure $\mu_n \otimes \mu_n$ converges weak* to $\mu \otimes \mu $; thus $\lim_n \mathcal{E}_r(\mu_n) = \mathcal{E}_r(\mu)$. It follows that $\mathcal{E}_r$ satisfies also $(A3).$ 

    Furthermore notice that by Equation \eqref{eqn:energy_formula},
    we have
    \[ \liminf_{n}\mathcal{E}(\mu_n) \geq \lim_n\mathcal{E}_r(\mu_n) = \mathcal{E}_r(\mu). \]
    Since this holds for every $r<1$ we conclude that also $\mathcal{E}$ satisfies $(A3)$.
\end{proof}

Since the underlying kernel is understood by the context we shall denote by $\ccap$ the capacity generated by the energy functional $\mathcal{E}$ and by $\ccap_r$ the \emph{$r$-capacity} generated by the energy functional $\mathcal{E}_r$.

\begin{lemma}\label{lem:c_r} 
For a compact set $F\subset  {\overline{\mathbb{B}}_d}$ it holds that \[\lim_{r\to 1}\ccap_r(F) = \ccap(F).\]
\end{lemma}

\begin{proof}
First notice that, since for every measure $\mu$, the energy $\mathcal{E}_r(\mu)$ is increasing in $r$, the limit $\lim_{r\to 1}\ccap_r(F)$ exists and 
$\ccap(F)\leq \lim_{r \to 1} \ccap_r(F)$.

For the other inequality, let $\mu_r$ be a probability measure supported in $F$ such that $\mathcal{E}_r(\mu_r)=\ccap_r(F)^{-1}$.  Such measures exist as a consequence of property $(A3)$ of energy functionals and weak* sequential compactness of $P(F)$.
Let $(r_n)$ be a sequence tending to $1$ such that $\mu_{r_n}$ tends weak* to
a measure $\mu \in P(F)$.
Let $0 <r < 1$.
Since $\mathcal E_{r}(\mu_{r_n}) \le \mathcal E_{r_n} (\mu_{r_n})$ for sufficiently large $n$, we have
\[ \mathcal{E}_r(\mu) =  \lim_n\mathcal{E}_{r}(\mu_{r_n}) \leq \lim_n\mathcal{E}_{r_n}(\mu_{r_n}) = \lim_n \ccap_{r_n}(F)^{-1}  = \lim_{r \to 1} \ccap_r(F)^{-1}.\]
By taking the limit in $r$ we find that 
\[ \lim_{r\to 1 }\ccap_r(F) \leq \frac{1}{\mathcal{E}(\mu)} \leq \ccap(F), \]
establishing the other inequality.
\end{proof}

\begin{lemma}\label{prop:potential_bound_below}
 Let $F \subset \overline{\mathbb B_d}$ be a non empty compact set, $0<r<1$ and $\mu$ an $r$-equilibrium measure of $F$.
 Define
\[ f(z) = \int_{\overline{\mathbb{B}}_d} K(rz,rw) d\mu(w). \]
Then,
\begin{enumerate}[label=\normalfont{(\alph*)}]
    \item $\Re f(z) \ge \ccap_r(F)^{-1}$ for all $z \in F$;
    \item $\|f\|_{\mathcal H}^2 \le \ccap_r(F)^{-1}$.
\end{enumerate}
\end{lemma}

\begin{proof}
Note that $\ccap_r(F) \neq 0$ as $F$ is not empty.

(a)
Let $z \in F$ and let $\delta_z$ be the Dirac measure at $z$.
Then by $(P6)$ of Proposition \ref{prop:abstract_capacity} we have 
\[
\ccap_r(F)^{-1}  = \mathcal{E}_r(\mu) \leq \mathcal{E}_r(\mu,\delta_z)
= \Re f(z).
\]

(b)
Let $\mathcal H_r$ be the space
with reproducing kernel $(z,w) \mapsto K(rz,rw)$.
Then $\mathcal H_r$ is contractively contained in $\mathcal H$,
so applying Proposition \ref{prop:energy_fa} to $\mathcal H_r$,
we find that
\[
    \|f\|_{\mathcal H}^2 \le \|f\|_{\mathcal H_r}^2
    = \mathcal E_r(\mu) = \ccap_r(F)^{-1},
\]
establishing (b).
\end{proof}

The following result in particular establishes Theorem \ref{thm:unboundedness_intro}.

\begin{theorem}
  \label{thm:cap_0_unboundedness}
  Let $\mathcal{H}$ be a unitarily invariant space
  of pluriharmonic functions on $\mathbb{B}_d$ whose kernel has non-negative real part.
  Let $E \subset \partial \mathbb{B}_d$ be compact with $\ccap_{\mathcal{H}}(E) = 0$.
  Then there exists $f \in \mathcal{H}$ with $\lim_{r \nearrow 1} |f(r \zeta)| = \infty$
  for all $\zeta \in E$.
\end{theorem}
 
\begin{proof}
  Lemma \ref{lem:c_r} shows that $\lim_{r\to 1}\ccap_r(F) = 0$. Therefore we can find a sequence $r_n\to 1 $ such that 
\[ \sum_{n=1}^\infty (\ccap_{r_n}(F))^{\frac12} < \infty.   \]
Consider a sequence of equilibrium measures $\mu_n$ for the $r_n$-capacity of $F$. We define the (normalized) potentials 
\[ f_n(z) = \ccap_{r_n}(F) \int_{\bnd} K(r_nz,r_nw) d\mu_n(w). \]
Then by Lemma \ref{prop:potential_bound_below},  we have that  
$\|f_n\|_{\mathcal H}^2 \le \ccap_{r_n}(F)$,
hence the series
\[ f:=\sum_{n=1}^\infty f_n \]
converges absolutely in $\mathcal H$.
Finally let $z \in E$. Then by Fatou's Lemma and Lemma \ref{prop:potential_bound_below},
\begin{align*} \liminf_{r\to 1} \sum_{n=1}^\infty \Re f_n(r  z)  & \geq \sum_{n=1}^\infty \liminf_{r\to 1} \Re f_n(r z) = \sum_{n=1}^\infty 1 = \infty, \end{align*}
which proves the theorem.
\end{proof}

\section{Weak type estimate of maximal function}
\label{sec:weak_type}

The goal of this section is to establish
Theorem \ref{thm:boundary_limits_intro}, i.e.\ the existence of \Kor\ limits of functions in $H^2_d$
outside of a set of capacity zero.

For $\alpha > 1$ and $\zeta \in \partial \mathbb{B}_d$, let
\begin{equation*}
  D_\alpha(\zeta) = \{z \in \mathbb{B}_d: | 1 - \langle z,\zeta \rangle| < \frac{\alpha}{2} (1 - \|z\|^2) \}.
\end{equation*}
By definition, a function $f \in C(\mathbb{B}_d)$ has \emph{\Kor\ limit} $\lambda$ at a point $\zeta \in \partial \mathbb{B}_d$ if
for every $\alpha > 1$, we have $\lim_{z \to \zeta, z \in D_{\alpha}(\zeta)} f(z) = \lambda$.
If $f \in C(\mathbb{B}_d)$, let
\begin{equation*}
  (M_\alpha f)(\zeta) = \sup_{z \in D_\alpha(\zeta)} |f(z)| \quad (\zeta \in \partial \mathbb{B}_d)
\end{equation*}
be the \emph{\Kor\ maximal function} of $f$, which is a lower semi-continuous function on $\partial \mathbb{B}_d$; cf.\ \cite[Definition 5.4.4]{Rudin08}.

As remarked in Section \ref{sec:energy_basics},
we will work with the pluriharmonic Drury--Arveson space $\PHDA$,
i.e.\ the RKHS
on $\mathbb{B}_d$ with kernel
\begin{equation*}
   2 \Re \frac{1}{1 - \langle z,w \rangle } - 1 = \frac{1 - |\langle z,w \rangle |^2}{|1 - \langle z,w \rangle |^2}. 
\end{equation*}

The key to proving Theorem \ref{thm:boundary_limits_intro} is the following weak-type
estimate for the \Kor\ maximal function.

\begin{theorem}
  \label{thm:weak_type}
  For each $\alpha > 1$, there exists a constant $C_\alpha$, depending only on $\alpha$, such that
  \begin{equation*}
    \ccap_{h^2_d}^*( \{ M_\alpha f > t \}) \le C_\alpha^2 \frac{\|f\|^2}{t^2} \quad \text{ for all } f \in \PHDA, t > 0.
  \end{equation*}
\end{theorem}

The proof of Theorem \ref{thm:weak_type} occupies most of the remainder of this section.
While we are mainly interested in $\PHDA$, we will initially work in somewhat greater generality.
Thus, let $\mathcal{H}$ be a unitarily invariant space of pluriharmonic functions with real-valued non-negative kernel $k$.

To prove the weak-type estimate, we will linearize the maximal function and apply a $T T^*$-argument. Both techniques are quite classical in harmonic analysis. An in depth study of linearizable operators, of which maximal functions are a particular example, can be found in \cite[Chapter V.1]{Francia1985}. The $TT^*$ method on the other hand belongs to the so called ``spectral methods'' of harmonic analysis, and as such it relies heavily on the Hilbert space structure of the underlying space.

Let $\psi: \partial \mathbb{B}_d \to \mathbb{B}_d$ be a measurable function with $\psi(\zeta) \in D_\alpha(\zeta)$ for all $\zeta \in \partial \mathbb{B}_d$.
We also assume that $\psi$ takes values in a compact subset of $\mathbb{B}_d$.
Let $\mu$ be a Borel probability measure on $\partial \mathbb{B}_d$.
Then the linear operator
\begin{equation*}
  T_\psi: \mathcal{H} \to L^1(\mu), \quad (T_\psi f)(\zeta) = f(\psi(\zeta))
\end{equation*}
is bounded.
The following lemma shows that controlling the maximal function is equivalent to controlling the operators $T_{\psi}$ uniformly
for all functions $\psi$.
As usual, the idea is to let $\psi(\zeta)$ be a point in $D_\alpha(\zeta)$ at which
the supremum in the definition of $M_{\alpha} f(\zeta)$ is (almost) achieved.
Since we have to ensure measurability of $\psi$, the actual proof is slightly
more technical.

\begin{lemma}
  \label{lem:linearization}
  Let $\mu$ be a Borel probability measure on $\partial \mathbb{B}_d$ and let $f \in C(\mathbb{B}_d)$. Then
  \begin{equation*}
    \|M_\alpha f\|_{L^1(\mu)} = \sup_{\psi} \|T_\psi f\|_{L^1(\mu)},
  \end{equation*}
  where the supremum is taken over all measurable functions $\psi: \partial \mathbb{B}_d \to \mathbb{B}_d$
  that satifsy $\psi(\zeta) \in D_\alpha(\zeta)$ for all $\zeta \in \partial \mathbb{B}_d$
  and that take values in a compact subset of $\mathbb{B}_d$.
\end{lemma}

\begin{proof}
  The pointwise bound
  \begin{equation*}
    |T_\psi f(\zeta)| \le |M_\alpha f(\zeta)| \quad \text{ for all } \zeta \in \partial \mathbb{B}_d
  \end{equation*}
  and all admissible functions $\psi$ shows that $\sup_{\psi} \|T_\psi f\|_{L^1(\mu)} \le \|M_\alpha f\|_{L^1(\mu)}$.

  To prove the converse inequality, we fix a point $\xi \in \partial \mathbb{B}_d$ and a countable
  dense subset $D \subset D_\alpha(\xi)$. If $\mathcal{U}(d)$ denotes the group of $d \times d$ unitaries,
  then the surjective map
  \begin{equation*}
    \Phi: \mathcal{U}(d) \to \partial \mathbb{B}_d, \quad U \mapsto U \xi,
  \end{equation*}
  admits a Borel right-inverse $\Psi: \partial \mathbb{B}_d \to \mathcal{U}(d)$.
  For instance, this follows from the fact that $\Phi$ is a fiber bundle,
  or from the abstract result that every continuous surjection
  between compact metric spaces admits a Borel right-inverse; see e.g.\ \cite[Corollary 5.2.6]{Srivastava98}.
  We write $U_\zeta = \Psi(\zeta)$, thus $U_\zeta \xi = \zeta$
  and hence $U_\zeta D_\alpha(\xi) = D_\alpha(\zeta)$ for all $\zeta \in \partial \mathbb{B}_d$.
  By continuity of $f$ on $\mathbb{B}_d$, it follows that
  \begin{equation*}
    (M_\alpha f)(\zeta) = \sup_{z \in U_\zeta D} |f(z)| \quad \text{ for all } \zeta \in \partial \mathbb{B}_d.
  \end{equation*}
  Let $D = \{d_1,d_2,\ldots \}$ and $D_N = \{d_1,\ldots,d_N\}$. Define
  \begin{equation*}
    (M_N f)(\zeta) = \sup_{z \in U_\zeta D_N} |f(z)|.
  \end{equation*}
  Then $(M_N f)_N$ increases to $M_\alpha f$, so by the monotone convergence theorem,
  \begin{equation}
    \label{eqn:monotone_convergence}
    \|M_\alpha f\|_{L^1(\mu)} = \sup_{N \in \mathbb{N}} \|M_N f\|_{L^1(\mu)}.
  \end{equation}
Let us now fix $N \in \mathbb{N}$; we will realize
$M_N f$ with a measurable function $\psi$. For each $j=1,\ldots,N$, let
  \begin{equation*}
    A_j = \{\zeta \in \partial \mathbb{B}_d: (M_N f)(\zeta) = |f( U_\zeta d_j)| \}.
  \end{equation*}
  Then each $A_j$ is a Borel subset of $\partial \mathbb{B}_d$.
  Let $B_1 = A_1, B_2 = A_2 \setminus A_1, B_3 = A_3 \setminus (A_1 \cup A_2)$ and so on.
  Then $B_1,\ldots,B_N$ form a Borel partition of $\partial \mathbb{B}_d$. Let
  \begin{equation*}
    \psi: \partial \mathbb{B}_d \to \mathbb{B}_d, \quad \zeta \mapsto U_\zeta d_j
    \quad \text{ if } \zeta \in B_j.
  \end{equation*}
  Then $\psi$ is Borel measurable,
  $\psi(\zeta) \in D_\alpha(\zeta)$ for all $\zeta \in \mathbb{B}_d$,
  $\psi$ takes values in a compact subset of $\mathbb{B}_d$ and
  $(M_N f)(\zeta) = |f(\psi(\zeta))| $ for all $\zeta \in \mathbb{B}_d$.
  Hence
  \begin{equation*}
    \|M_N f\|_{L^1(\mu)} = \|T_\psi f\|_{L^1(\mu)}.
  \end{equation*}
  Combining this identity with \eqref{eqn:monotone_convergence}, the result follows.
\end{proof}

The norm of the linear operators $T_\psi$ can be computed expicitly
with the help of duality.

\begin{lemma}
  \label{lem:TTstar}
  Let $\mathcal{H}$ be a unitarily invariant space of pluriharmonic functions with non-negative kernel $k$.
  For each measurable function $\psi: \partial \mathbb{B}_d \to \mathbb{B}_d$
  taking values in a compact subset of $\mathbb{B}_d$, we have
  \begin{equation*}
    \|T_\psi\|_{\mathcal{H} \to L^1(\mu)}^2 = \int_{\partial \mathbb{B}_d} \int_{\partial \mathbb{B}_d} {k}(\psi(\zeta),\psi(\eta)) \, d \mu(\eta) d \mu(\zeta).
  \end{equation*}
\end{lemma}

\begin{proof}
  Note that $T_\psi^*$ is an operator from $L^\infty(\mu)$ into $\mathcal{H}$, where we identify $L^\infty(\mu)$ with the conjugate dual of $L^1(\mu)$. Moreover, 
  \begin{equation*}
    \|T_\psi\|^2_{\mathcal{H} \to L^1(\mu)} = \|T_\psi T_\psi^*\|_{L^\infty(\mu) \to L^1(\mu)}.
  \end{equation*}

  Now, for all $z \in \mathbb{B}_d$ and $g \in L^\infty(\mu)$, we have
  \begin{align*}
  (T_\psi^* g)(z) = \langle T_\psi^* g, {k}(\cdot,z) \rangle_{\mathcal{H}}
   &=\int_{\partial \mathbb{B}_d} g(\eta) \overline{T_\psi( {k}(\cdot,z))(\eta)} \, d \mu(\eta) \\
   &= \int_{\partial \mathbb{B}_d} g(\eta) {k}(z, \psi(\eta)) \, d \mu(\eta).
  \end{align*}
  Therefore,
  \begin{equation*}
    (T_\psi T_\psi^* g)(\zeta)
  = \int_{\partial \mathbb{B}_d} g(\eta) {k}(\psi(\zeta), \psi(\eta)) \, d \mu(\eta).
  \end{equation*}
  Thus, if $\|g\|_{L^\infty(\mu)} = 1$, then since ${k}$ is non-negative, we have
  \begin{equation*}
    \|T_\psi T_\psi^* g\|_{L^1(\mu)} \le \int_{\partial \mathbb{B}_d} \int_{\partial \mathbb{B}_d} {k}(\psi(\zeta), \psi(\eta)) \, d \mu(\eta) d \mu(\zeta).
  \end{equation*}
  Equality holds for $g = 1$, which gives the desired formula for the operator norm of $T_\psi T_\psi^*$.
\end{proof}

We require the following basic estimate, which is related to classical estimates for the invariant Poisson
kernel; see e.g.\ \cite[Lemma 5.4.3]{Rudin08}.
As usual, we will write $A \lesssim B$ to mean that there exists a constant $C \in (0,\infty)$
such that $A \le C B$.

\begin{lemma}
  \label{lem:Koranyi_estimates}
  Let $\zeta \in \partial \mathbb{B}_d$, let $w \in \overline{\mathbb{B}_d}$ and let $\psi(\zeta) \in D_\alpha(\zeta)$.
  Then
  \begin{equation*}
    |1 - \langle \zeta, w \rangle| \lesssim |1 - \langle \psi(\zeta), w \rangle|,
  \end{equation*}
  where the implied constant only depends on $\alpha$.
\end{lemma}

\begin{proof}
  Let $d(a,b) = |1 - \langle a,b \rangle|^{1/2}$ for $a ,b \in \overline{\mathbb{B}_d}$
  denote the \Kor\ distance, which satisfies the triangle inequality; see e.g.\ \cite[Proposition 5.1.2]{Rudin08}.
  Then
  \begin{equation*}
    d(\zeta,w) \le d(\zeta,\psi(\zeta)) + d(\psi(\zeta),w).
  \end{equation*}
  By definition of $D_\alpha(\zeta)$, we have
  \begin{align*}
    | 1 - \langle \zeta,\psi(\zeta) \rangle|
    < \frac{\alpha}{2} (1 - \|\psi(\zeta)\|^2)
    \lesssim 1 - \|\psi(\zeta)\|
    &\le 1 - \|\psi(\zeta)\| \|w\| \\
    &\le |1 - \langle \psi(\zeta),w \rangle|.
  \end{align*}
  Combining both inequalities yields the result.
\end{proof}

In the case of standard weighted Dirichlet spaces, the preceding results fairly easily imply a key estimate for
the maximal function, since the real part and the modulus of the reproducing kernel are comparable.
The result is known at least for the Dirichlet space in the unit disc \cite[p.34]{EKM+14}, but the proof below may be new.

\begin{proposition}
  \label{prop:maximal_function_weighted_Dirichlet}
  Let
  \begin{equation*}
    K(z,w) = \log \Big( \frac{e}{1 - \langle z,w \rangle} \Big)
  \end{equation*}
  or
  \begin{equation*}
    K(z,w) = \frac{1}{(1 - \langle z,w \rangle)^a }
  \end{equation*}
  for $a \in (0,1)$. Let $\mathcal{H}$ be the RKHS with reproducing kernel $\Re K $ .
  Then for all $\alpha > 1$, there exists a constant $C_\alpha$ such that for all $\mu \in P(\partial \mathbb{B}_d)$ and all $f \in \mathcal{H}$, we have
  \begin{equation*}
    \|M_\alpha f\|_{L^1(\mu)} \le C_\alpha \mathcal{E}_{\mathcal{H}}(\mu)^{1/2} \|f\|_{\mathcal{H}}.
  \end{equation*}
\end{proposition}

\begin{proof}
  Let $k = \Re K$.
  Let $\psi: \partial \mathbb{B}_d \to \mathbb{B}_d$ be a measurable function that satisfies
  $\psi(\zeta) \in D_\alpha(\zeta)$ for all $\zeta \in \partial \mathbb{B}_d$ and that takes values
  in a compact subset of $\mathbb{B}_d$.
  We show that the linearized operator $T_\psi$ is bounded as an operator $\mathcal{H} \to L^1(\mu)$
  and that 
  \begin{equation*}
    \|T_\psi\|_{\mathcal{H} \to L^1(\mu)} \le C_\alpha \mathcal{E}_\mathcal{H}(\mu)^{1/2},
  \end{equation*}
  where $C_\alpha$ is a constant that only depends on $\alpha$. The result then follows
  from Lemma \ref{lem:linearization}.

  By Lemma \ref{lem:TTstar}, we have
  \begin{equation*}
    \|T_\psi\|_{\mathcal{H} \to L^1(\mu)}^2 = \int_{\partial \mathbb{B}_d} \int_{\partial \mathbb{B}_d} {k}(\psi(\zeta),\psi(\eta)) \, d \mu(\eta) d \mu(\zeta),
  \end{equation*}
  so it suffices to bound the right-hand side by a constant times $\mathcal{E}_{\mathcal H}(\mu)$.
  Applying Lemma \ref{lem:Koranyi_estimates} twice, we find that
  \begin{equation}
    \label{eqn:Koranyi_estimate}
    |1 - \langle \zeta,\eta \rangle| \lesssim |1 - \langle \psi(\zeta),\eta \rangle|
    \lesssim |1 - \langle \psi(\zeta), \psi(\eta) \rangle|.
  \end{equation}
  If $K(z,w) = (1 - \langle z,w \rangle)^{-a}$, then $ k = \Re K$ and $|K|$ are comparable,
  so \eqref{eqn:Koranyi_estimate} implies that
  \begin{equation}
    \label{eqn:Koranyi_energy}
    \int_{\partial \mathbb{B}_d} \int_{\partial \mathbb{B}_d} k(\psi(\zeta),\psi(\eta)) \, d \mu(\eta) d \mu(\zeta)
    \lesssim \int_{\partial \mathbb{B}_d} \int_{\partial \mathbb{B}_d} k(\zeta,\eta) \,d \mu(\eta) d \mu(\zeta).
  \end{equation}
  The last expression is less or equal than $\mathcal{E}_\mathcal{H}(\mu)$ by Fatou's Lemma.

  Similarly, if $K(z,w) = \log( \frac{e}{1 - \langle z,w \rangle })$, then using
  that the energy of a probability measure is at least $1$, we find that \eqref{eqn:Koranyi_estimate} implies \eqref{eqn:Koranyi_energy},
  which again gives the desired estimate in terms of the energy.
\end{proof}

Proposition \ref{prop:maximal_function_weighted_Dirichlet} easily implies the weak-type estimate
for capacity in the (weighted) Dirichlet space, cf.\ the proof of Theorem \ref{thm:weak_type} below.

The argument in the last proof does not carry over to the Drury--Arveson space, since
the real part and the modulus of $\frac{1}{1 - \langle z,w \rangle }$ are not comparable.
Instead, we will use a more delicate argument.
We begin with the following refinement of Lemma \ref{lem:Koranyi_estimates}.

\begin{lemma}
  \label{lem:Koranyi_estimate_2}
  Let
  \begin{equation*}
    k(z,w)= 2 \Re \frac{1}{1 - \langle z,w \rangle } - 1 = \frac{1 - |\langle z,w \rangle |^2}{|1 - \langle z,w \rangle |^2}.
  \end{equation*}
  Let $\zeta \in \partial \mathbb{B}_d$, let $w \in \mathbb{B}_d$ and let $\psi(\zeta) \in D_\alpha(\zeta)$ with $\|\psi(\zeta)\| \ge \|w\|$.
  Then
  \begin{equation*}
    {k}(\psi(\zeta),w) \approx {k}(\zeta,w),
  \end{equation*}
  where the implied constants only depend on $\alpha$.
\end{lemma}

\begin{proof}
  We first show that
  \begin{equation}
    \label{eqn:koranyi_equiv}
    |1 - \langle \zeta, w \rangle| \approx |1 - \langle \psi(\zeta),w \rangle|.
  \end{equation}
  The upper bound was established in Lemma \ref{lem:Koranyi_estimates}.
  For the lower bound, we also use the triangle inequality for the \Kor\ distance, which gives
  \begin{equation*}
    d(\psi(\zeta),w) \le d(\psi(\zeta),\zeta) + d(\zeta,w).
  \end{equation*}
  Using first the definition of $D_\alpha(\zeta)$ and then the hypothesis $\|\psi(\zeta)\| \ge \|w\|$,
  we find that
  \begin{equation*}
    |1 - \langle \psi(\zeta),\zeta \rangle| \lesssim 1 - \|\psi(\zeta)\| \le 1 - \|w\| \le |1 - \langle \zeta,w \rangle|.
  \end{equation*}
  This establishes \eqref{eqn:koranyi_equiv}.

  Now,
  \begin{equation*}
    \frac{k(\psi(\zeta),w)}{k(\zeta,w)}  =
  \frac{1 - | \langle \psi(\zeta),w \rangle|^2 }{ 1 - | \langle  \zeta,w \rangle|^2}
    \cdot
    \frac{|1 - \langle \zeta,w \rangle|^2 }{ |1 - \langle \psi(\zeta),w \rangle|^2 }.
  \end{equation*}
  The second factor is bounded above and below thanks to \eqref{eqn:koranyi_equiv}, so it remains to bound the first factor.
  To this end, choose a unimodular $\omega \in \mathbb{C}$ such that $\omega \langle \zeta,w \rangle = |\langle \zeta,w \rangle|$.
  Then by the reverse triangle inequality,
  \begin{equation*}
    \frac{1 - | \langle \psi(\zeta),w \rangle|^2 }{ 1 - | \langle \zeta,w \rangle|^2 }
    \approx
    \frac{1 - | \langle \psi(\zeta),w \rangle| }{ 1 - | \langle \zeta,w \rangle| }
    \le \frac{ | 1 - \omega \langle \psi(\zeta), w \rangle|}{| 1 - \omega \langle \zeta,w \rangle|},
  \end{equation*}
  which is uniformly bounded thanks to \eqref{eqn:koranyi_equiv}, applied with $w$ replaced by $\overline{\omega} w$.
  For the lower bound, we instead choose $\omega$ so that $\omega \langle \psi(\zeta),w \rangle = |\langle \psi(\zeta),w \rangle|$.
\end{proof}

We can now establish the following analogue of Proposition \ref{prop:maximal_function_weighted_Dirichlet}.
It is the key step in the proof of the weak-type inequality for the maximal function.

\begin{proposition}
  \label{prop:maximal_function_L1}
  Let $\alpha > 1$. There exists a constant $C_\alpha$, depending only on $\alpha$, such that
  for all $\mu \in P(\partial \mathbb{B}_d)$ and for all $f \in \PHDA$, we have
  \begin{equation*}
    \|M_\alpha f\|_{L^1(\mu)} \le C_\alpha \mathcal{E}(\mu)^{1/2} \|f\|_{\PHDA}.
  \end{equation*}
\end{proposition}

\begin{proof}
  As before, we consider the linearized operator $T_\psi: \PHDA \to L^1(\mu)$,
  where $\psi: \partial \mathbb{B}_d\to \mathbb{B}_d $ is measurable, takes values in a compact subset of $\mathbb{B}_d$
  and satisfies $\psi(\zeta) \in D_\alpha(\zeta)$ for all $\zeta \in \partial \mathbb{B}_d$.
  We will show that the operator norm of $T_\psi$ is bounded by $C_\alpha \mathcal{E}(\mu)^{1/2}$, where $C_\alpha$ is a constant that only depends on $\alpha$.

  Let $k$ be the reproducing kernel of $\PHDA$.
  To estimate the operator norm of $T_\psi$, we apply Lemma \ref{lem:TTstar} to find that
  \begin{equation*}
    \|T_\psi\|_{\PHDA \to L^1(\mu)}^2   = \int_{\partial \mathbb{B}_d} \int_{\partial \mathbb{B}_d} k(\psi(\zeta),\psi(\eta)) \, d \mu(\eta) d \mu(\zeta).
  \end{equation*}
  We break up the integral into the regions $A =\{ \|\psi(\zeta)\| \ge \|\psi(\eta)\| \}$ and $ B = \{ \|\psi(\zeta)\| < \|\psi(\eta)\| \}$.
  Lemma \ref{lem:Koranyi_estimate_2} shows that
  \begin{align*}
    \iint_A k(\psi(\zeta), \psi(\eta)) d \mu(\eta) d \mu(\zeta)
    &\lesssim \iint_A k( \zeta, \psi(\eta)) d \mu(\eta) d  \mu(\zeta) \\
    &\le \int_{\partial \mathbb{B}_d} \int_{\partial \mathbb{B}_d} k(\psi(\eta), \zeta) \, d \mu(\zeta) d \mu(\eta),
  \end{align*}
  where the implied constant only depends on $\alpha$. By symmetry, the same estimate holds for the integral over $B$.
  Hence, recalling that $f_\mu \in \PHDA$ denotes the potential of $\mu$, we find that
  \begin{align*}
    \|T_\psi\|_{\PHDA \to L^1(\mu)}^2 &\lesssim \int_{\partial \mathbb{B}_d} \int_{ \partial \mathbb{B}_d} k(\psi(\eta), \zeta) \, d \mu(\zeta) d \mu(\eta) \\
                                               &= \Big| \int_{\partial \mathbb{B}_d} f_\mu(\psi(\eta)) \, d \mu(\eta) \Big| \\
                                               &\le \|T_\psi\|_{\PHDA \to L^1(\mu)} \|f_\mu\|_{\PHDA}.
  \end{align*}
  Proposition \ref{prop:energy_fa} shows that $\|f_\mu\|_{\PHDA} = \mathcal{E}(\mu)^{1/2}$,
  so we conclude that
  \begin{equation*}
    \|T_\psi\|_{\PHDA \to L^1(\mu)} \lesssim \mathcal{E}(\mu)^{1/2}.
  \end{equation*}
  The result now follows from Lemma \ref{lem:linearization}.
\end{proof}

With Proposition \ref{prop:maximal_function_L1} in hand, the proof of the weak-type inequality now proceeds in a standard manner; cf.\ the proof of \cite[Theorem 3.2.5]{EKM+14}.

\begin{proof}[Proof of Theorem \ref{thm:weak_type}]
  Let $t > 0$ and let $f \in \PHDA$. Let $F$ be a compact subset of $\{M_\alpha f > t \}$
  and let $\mu$ be a Borel probability measure supported on $F$. Then by Proposition \ref{prop:maximal_function_L1},
  \begin{equation*}
    t \le \int_{\partial \mathbb{B}_d} M_\alpha f d \mu \le C_\alpha \mathcal{E}(\mu)^{1/2} \|f\|_{\PHDA}.
  \end{equation*}
  Thus,
  \begin{equation*}
    \ccap(F) = \sup \Big\{ \frac{1}{\mathcal{E}(\mu)} : \mu \in \mathcal{P}(F) \Big\}
    \le C_\alpha^2 \frac{\|f\|^2_{\PHDA}}{t^2}.
  \end{equation*}
  Since $F$ was an arbitrary compact subset of $\{M_\alpha f > t\}$, the result follows from the definition of (inner) capacity;
  the inner capacity agrees with the outer capacity since $\{M_\alpha f > t \}$ is open.
\end{proof}

The existence of \Kor\ limits now follows from the weak-type estimate in the usual way,
since $\PHDA$ contains a dense subspace of functions that are continuous on $\overline{\mathbb{B}_d}$.
More explicitly, we have the following result, which implies Theorem \ref{thm:boundary_limits_intro}
since $H^2_d \subset \PHDA$ and since the sets of capacity zero agree for both spaces.

\begin{theorem}
  \label{thm:h^2_d_boundary}
  For each $f \in \PHDA$, there exists a Borel set $E \subset \partial \mathbb{B}_d$ with $\ccap^*_{h^2_d}(E) = 0$
  such that $f$ has a \Kor\ limit at every point in $\partial \mathbb{B}_d \setminus E$.
\end{theorem}

\begin{proof}
  By decomposing $f$ into real and imaginary part and using (P2) of Proposition \ref{prop:abstract_capacity}, it suffices to consider the case of real valued functions.
  Let $\alpha > 1$.
  For real-valued $f \in \PHDA$ define
  \begin{equation*}
    \beta_f(\zeta) = \limsup_{z \to \zeta, z \in D_\alpha(\zeta)} f(z) - \liminf_{z \to \zeta, z \in D_{\alpha}(\zeta)} f(z).
  \end{equation*}
  Just like the maximal function, the function $\beta_f$ is Borel.
  Let $t > 0$.
  We will show that $\{\beta_f > t\}$ has (outer) capacity zero.
  To this end, note that for each $0 \le r < 1$, the function $f_r$ is continuous on $\overline{\mathbb{B}_d}$, so
  $\beta_f = \beta_{f - f_r}$. Moreover, $\{\beta_{f - f_r} > t \} \subset \{2 M_\alpha (f - f_r) > t \}$,
  and the set on the right is open. Hence Theorem \ref{thm:weak_type} implies that
  \begin{equation*}
    \ccap^*_{h^2_d} \{\beta_f > t \} = \ccap^*_{h^2_d} \{\beta_{f - f_r} > t \} \le \ccap^*_{h^2_d} \{ 2 M_\alpha (f - f_r) > t \} \lesssim \frac{\|f - f_r\|^2}{t^2}.
  \end{equation*}
  Letting $r \to 1$, we find that $\ccap^*_{h^2_d} \{ \beta_f > t \} = 0$
  by Lemma \ref{lem:pluriharmonic_RKHS_basic} (c).
  Recall that $\beta_f$ depends on $\alpha$ and define $E_\alpha = \{ \beta_f \neq 0 \}$.
  Then $\ccap^*_{h^2_d}(E_\alpha) = 0$ by $(P2)$ of Proposition \ref{prop:abstract_capacity},
  and for every $\zeta  \in \partial \mathbb{B}_d \setminus E_\alpha$, the limit $\lim_{z \to \zeta, z \in D_\alpha(\zeta)} f(z)$ exists.
  Finally, let $E = \bigcup_{\alpha > 1} E_\alpha$. Then $\ccap^*_{h^2_d}(E) = 0$,
  and $f$ has \Kor\ limits at all points $\zeta \in \partial \mathbb{B}_d \setminus E$.
\end{proof}

\begin{remark}
\label{rem:AC_new_proof}
    Using Proposition \ref{prop:maximal_function_weighted_Dirichlet} in place of Proposition \ref{prop:maximal_function_L1} and arguing as the proof of Theorem \ref{thm:weak_type}, we obtain another proof of the weak type inequality
    for the maximal function in the space $\mathcal D_a$ on the ball.
    Exactly as in the proof of Theorem \ref{thm:h^2_d_boundary}, this then
    yields the Hilbert space case of a theorem of Ahern and Cohn \cite{AC89},
    namely
    that every function in $\mathcal D_a$ has a \Kor\ limit outside
    of a Borel set of outer $\mathcal D_a$-capacity zero.
\end{remark}

\section{Radial limits in more general spaces}

In this section, we establish another result about the existence of radial limits.
Compared to Theorem \ref{thm:h^2_d_boundary}, it applies to every unitarily invariant space of pluriharmonic functions whose kernel has non-negative real part, but only yields a weaker notion of convergence.

We begin with the following application of Lemma \ref{lem:TTstar}.

\begin{lemma}
  \label{lem:L^1_energy}
  Let $\mathcal{H}$ be a unitarily invariant space of pluriharmonic functions with non-negative kernel on $\mathbb{B}_d$
  and let $\mu \in P(\partial \mathbb{B}_d)$ with $\mathcal{E}(\mu) < \infty$. Then for every $f \in \mathcal{H} \cap C(\overline{\mathbb{B}_d})$,
  we have
  \begin{equation*}
    \|f\|_{L^1(\mu)} \le \mathcal{E}(\mu)^{1/2} \|f\|_{\mathcal{H}}.
  \end{equation*}
\end{lemma}

\begin{proof}
  Let $0 \le r < 1$ and let
  \begin{equation*}
    T_r: \mathcal{H} \to L^1(\mu), \quad f \mapsto f_r.
  \end{equation*}
  Then $T_r$ is bounded and Lemma \ref{lem:TTstar} shows that
  \begin{equation*}
    \|T_r\|^2 = \int_{\partial \mathbb{B}_d} \int_{\partial \mathbb{B}_d} k(r z, r w) \, d \mu(w) d \mu(z)
    \le \mathcal{E}(\mu).
  \end{equation*}
  Since $f_r$ converges to $f$ uniformly on $\partial \mathbb{B}_d$, we conclude that
  \begin{equation*}
    \|f\|_{L^1(\mu)} = \lim_{r \nearrow 1} \|f_r\|_{L^1(\mu)} \le \sup_{0 \le r < 1} \|T_r\| \|f\|_{\mathcal{H}}
      \le \mathcal{E}(\mu)^{1/2} \|f\|_{\mathcal{H}}. \qedhere
  \end{equation*}
\end{proof}

We can now establish the announced result about radial limits.

\begin{theorem}
  \label{thm:convergence_in_capacity}
  Let $\mathcal{H}$ be a unitarily invariant space of pluriharmonic functions
  on $\mathbb B_d$ whose kernel has non-negative real part
  and let $f \in \mathcal{H}$.
  Then there exists a Borel measurable function $f^*: \partial \mathbb{B}_d \to \mathbb{C}$
  such that \emph{$(f_r)$ converges to $f^*$ in capacity} as $r \nearrow 1$, i.e.\ for every $\varepsilon > 0$,
  we have
  \begin{equation*}
    \ccap^*_{\mathcal{H}} ( \{|f^* - f_r| \ge \varepsilon \}) \to 0 \quad \text{ as } r \nearrow 1.
  \end{equation*}
  The function $f^*$ is uniquely determined outside of a Borel set of outer capacity zero.
  Moreover, there exists a sequence $r_n$ tending to $1$ such that $(f_{r_n})$ converges
  to $f^*$ pointwise outside of a Borel set of outer capacity zero.
\end{theorem}

\begin{proof}
  By replacing the kernel with its real part, we may assume without loss of generality that the kernel is real valued.
  (Note that $\mathcal H$ embeds into the space whose kernel is the real part of the original kernel, and the capacities are the same.)
  
  We show that $(f_r)$ satisfies a Cauchy condition with respect to convergence in capacity.
  Let $\varepsilon > 0$ and let $0 \le r,s < 1$.
  Let $\mu$ be a Borel probability measure supported on the compact set $\{|f_s - f_r| \ge \varepsilon \}$.
  Then Lemma \ref{lem:L^1_energy} shows that
  \begin{equation*}
    \varepsilon \le \int_{\partial \mathbb{B}_d} |f_s - f_r| \,d \mu
    \le \mathcal{E}(\mu)^{1/2} \|f_s - f_r\|_{\mathcal{H}}.
  \end{equation*}
  By definition of capacity, this means that
  \begin{equation*}
    \ccap_{\mathcal{H}} ( \{ |f_s - f_r| \ge \varepsilon \})
    \le \frac{\|f_s - f_r\|_{\mathcal{H}}^2}{\varepsilon^2}.
  \end{equation*}
  By capacitability of compact sets ($(P4)$ of Proposition \ref{prop:abstract_capacity}),
  the same estimate holds for the outer capacity. It follows that
  \begin{equation}
    \label{eqn:Cauchy_capacity}
    \lim_{r,s \to 1} \ccap^*_{\mathcal{H}} \{ |f_s - f_r| \ge \varepsilon \} = 0
  \end{equation}
  for all $\varepsilon > 0$.
  In this setting, we obtain a function $f^*: \partial \mathbb{B}_d \to \mathbb{C}$
  and a sequence $(r_n)$ tending to $1$ such that $(f_{r_n})$
  tends to $f^*$ in capacity and pointwise outside of a set of outer capacity zero;
  see \cite{Meyer76}. Moreover, the function $f^*$ is uniquely determined outside of a set
  of outer capacity zero.
  (The proof of these particular assertions is essentially the same as that
  of the corresponding statements
  for convergence in measure, it only uses monotonicity and countable subadditivity of outer
  capacity; see e.g.\ \cite[Theorem 2.30]{Folland99}).
  Finally, for all $\varepsilon > 0$ and $0 \le r < 1$, we have
  \begin{equation*}
    \ccap_{\mathcal{H}}^* \{ |f^* - f_r| \ge \varepsilon \}
    \le \ccap_{\mathcal{H}}^* \Big\{ |f^* - f_{r_n}| \ge \frac{\varepsilon}{2} \Big\}
    + \ccap_{\mathcal{H}}^* \Big\{ |f_{r_n} - f_r| \ge \frac{\varepsilon}{2} \Big\}.
  \end{equation*}
  For sufficiently large $n$ and $r$ close to $1$, both summands
  can be made arbitrarily small (the second summand by \eqref{eqn:Cauchy_capacity}),
  which finishes the proof.
\end{proof}

The following result connects the boundary function $f^*$ to $L^1$ spaces
and the integration functionals $\rho_\mu$.
It will be useful in the next section.

\begin{proposition}
  \label{prop:boundary_L^1}
  Let $\mathcal{H}$ be a unitarily invariant space of pluriharmonic functions $\mathbb{B}_d$ whose kernel has non-negative real part.
  Let $\mu \in P(\partial \mathbb{B}_d)$ with $\mathcal{E}(\mu) < \infty$.
  If $f \in \mathcal{H}$, then $f^* \in L^1(\mu)$ with
  \begin{equation*}
    \|f^*\|_{L^1(\mu)} \le \mathcal{E}(\mu)^{1/2} \|f\|_{\mathcal{H}}.
  \end{equation*}
  Moreover,
  \begin{equation*}
    \rho_\mu(f) = \int_{\partial \mathbb{B}_d} f^* \,d \mu.
  \end{equation*}
\end{proposition}

\begin{proof}
  We may again assue that the kernel is real valued.
  By Theorem \ref{thm:convergence_in_capacity}, there exists a Borel set $E \subset \partial \mathbb{B}_d$
  with $\ccap_{\mathcal{H}}^*(E) = 0$ and a sequence $(r_n)$ converging to $1$
  such that $f_{r_n}$ converges to $f^*$ pointwise on $\partial \mathbb{B}_d \setminus E$.
  Lemma \ref{lem:capacity_zero} implies that
  $\mu(E) = 0$, so $(f_{r_n})$ converges to $f^*$ almost everywhere with respect to $\mu$.
  On the other hand, an application of Lemma \ref{lem:L^1_energy} shows that
  $(f_{r_n})$ is Cauchy in $L^1(\mu)$. Thus, $(f_{r_n})$ converges to $f^*$ in $L^1(\mu)$.
  The estimate for $\|f^*\|_{L^1(\mu)}$ also follows from Lemma \ref{lem:L^1_energy}.
  Moreover, since $\rho_\mu$ is continuous by Proposition \ref{prop:energy_fa},
  we have
  \begin{equation*}
    \rho_\mu(f) = \lim_{n \to \infty} \rho_\mu(f_{r_n}) = \lim_{n\to\infty} \int_{\partial \mathbb{B}_d} f_{r_n} \, d \mu
    = \int_{\partial \mathbb{B}_d} f^* \,d \mu. \qedhere
  \end{equation*}
\end{proof}

\section{Applications to cyclicity}
\label{sec:cyclic}

A \emph{regular unitarily invariant space of holomorphic
functions} on $\mathbb B_d$ is a reproducing kernel Hilbert space $\mathcal H$
of functions on $\mathbb B_d$ whose reproducing kernel is of the form
\[
    K(z,w) = \sum_{n=0}^\infty a_n \langle z,w \rangle^n,
\]
where $a_n > 0$ for all $n \ge 0$, $a_0 = 1$ and $\lim_{n \to \infty} \frac{a_n}{a_{n+1}} = 1$.
These conditions ensure, among other things, that the polynomials are multipliers
of $\mathcal H$; see e.g.\ \cite[Subsection 2.1]{BHM18} for further discussion
of such spaces.

Let $\mathcal{H}$ be a regular unitarily invariant space of holomorphic functions on $\mathbb{B}_d$
whose kernel has non-negative real part.
Given $f \in \mathcal{H}$, we let
\begin{equation*}
  [f] = \overline{\{ \varphi f: \varphi \in \Mult(\mathcal{H}) \}} \subset \mathcal{H}
\end{equation*}
be the smallest multiplier invariant subspace of $\mathcal{H}$ containing $f$.
The function $f$ is said to be \emph{cyclic} if $[f] = \mathcal{H}$.

Since the polynomials are SOT-dense in $\Mult(\mathcal{H})$
(see e.g.\ the proof of \cite[Lemma 3.1]{BHM18}),
we have
\begin{equation*}
  [f] = \overline{ \{ p f: p \text{ polynomial} \}}.
\end{equation*}
Moreover, since $\Mult(\mathcal{H})$ is dense in $\mathcal{H}$, we see that $f$ is cyclic
if and only if $1 \in [f]$.

In the classical Dirichlet space $\mathcal{D}$, it is well known that a necessary condition
for cyclicity of a function $f \in \mathcal{D}$ is that the set 
\begin{equation*}
  \{\zeta \in \mathbb{D}: f^*(\zeta) = 0 \}
\end{equation*}
has logarithmic
capacity zero; see \cite[Theorem 5]{BS84} or \cite[Corollary 9.2.5]{EKM+14}.
Indeed, the functional analytic point of view on capacity shows that if $f \in \mathcal{H}$
extends to be continuous on $\overline{\mathbb{B}_d}$ and the zero set
\begin{equation*}
  \{\zeta \in \partial \mathbb{B}_d : f(\zeta) = 0 \}
\end{equation*}
has positive $\mathcal{H}$-capacity, then the zero set supports a probability
measure $\mu$ of finite $\mathcal{H}$-energy.
By Proposition \ref{prop:energy_fa}, the integration
functional $\rho_\mu$ is then a non-zero continuous linear functional, whose
kernel contains $p f$ for every polynomial $p$; whence $f$ is not cyclic.

The observation in the last paragraph can be generalized to functions in $\mathcal{H}$
that do not extend to be continuous on $\overline{\mathbb{B}_d}$, but some
care must be taken with the precise meaning of vanishing on subsets of the boundary.

Given $f \in \mathcal{H}$, we let $f^*: \partial \mathbb{B}_d \to \mathbb{C}$ denote the boundary function of Theorem \ref{thm:convergence_in_capacity}; if $\mathcal H = H^2_d$,
then $f^*$ agrees with the boundary function given by Theorem \ref{thm:h^2_d_boundary}
outside of a set of outer capacity zero.
We also write
\begin{equation*}
  Z(f^*) = \{\zeta \in \partial \mathbb{B}_d: f^*(\zeta) = 0 \}.
\end{equation*}
Given an arbitrary subset $E \subset \partial \mathbb{B}_d$, we say that $f^* = 0$
\emph{quasi-everywhere on $E$} if $\ccap^*(E \setminus Z(f^*)) = 0$.
(Since the boundary function $f^*$ is uniquely determined outside of a set of outer capacity zero,
different choices of boundary function can at most lead to sets $Z(f^*)$ differing by a set of outer capacity  zero).
We write
\begin{equation*}
  \mathcal{H}_E = \{f \in \mathcal{H}: f^* = 0 \text{ quasi-everywhere on } E \}.
\end{equation*}
If $\mathcal{H}$ is the Dirichlet space on the disc, then each $\mathcal{H}_E$ is a closed
multiplier invariant subspace; see \cite[Theorem 9.2.3]{EKM+14}.
We can now establish the Drury--Arveson space analogue of this fact.

\begin{theorem}
  \label{thm:inv_sub_quasi_everywhere}
  For every subset $E \subset \partial \mathbb{B}_d$, the space $(H^2_d)_E$
  is a closed multiplier invariant subspace of $H^2_d$.
\end{theorem}

\begin{proof}
  The proof is essentially identical to that of \cite[Theorem 9.2.3]{EKM+14} using the weak-type
  inequality for capacity. Since the argument is short, we provide it for the convenience of the reader.

  It is clear that $(H^2_d)_E$ is a subspace and that $f \in (H^2_d)_E$ implies $p f \in (H^2_d)_E$ for all polynomials $p$.
  Since polynomials are SOT-dense in $\Mult(H^2_d)$, it suffices to show that $\mathcal{H}_E$ is closed.
  
  Let $(f_n)$ be a sequence in $(H^2_d)_E$ that converges to $f$ in $H^2_d$ and let $t >0$. Then by the weak-type inequality
  (Theorem \ref{thm:weak_type}), we have
  \begin{equation*}
  \ccap^*( E \cap \{|f^*| > t \}) \le \ccap^*( |f^* - f_n^*| > t \}) \lesssim \frac{\|f - f_n\|^2}{t^2}.
  \end{equation*}
  Taking the limit $n \to \infty$ and using that countable unions of sets of outer capacity zero
  have outer capacity zero, it follows that $f \in (H^2_d)_E$.
\end{proof}

The following corollary is quite immediate.

\begin{corollary}
  \label{cor:cyclic_outer_cap}
  Let $f \in H^2_d$ be cyclic. Then $\ccap^*(Z(f^*)) = 0$.
\end{corollary}

\begin{proof}
  Let $E = Z(f^*)$. Then $f \in (H^2_d)_E$ and hence $[f] \subset (H^2_d)_E$
  by Theorem \ref{thm:inv_sub_quasi_everywhere}.
  If $\ccap^*(E) > 0$, then $1 \notin (H^2_d)_E$, so $f$ is not cyclic.
\end{proof}

\begin{remark}
Corollary \ref{cor:cyclic_outer_cap} is related to recent work of Aleman, Perfekt, Richter, Sundberg and Sunkes \cite{APR+24}. They show
that if $f \in H^2_d$ extends to be continuous on $\overline{\mathbb B_d}$
and if $Z(f^*)$ contains the diffeomorphic image of a cube of real dimension
at least $3$, then $f$ is not cyclic \cite[Theorem 4.8]{APR+24}. The proof of their
result shows that in this case, $Z(f^*)$ has positive capacity.
In fact, $Z(f^*)$ even has positive capacity with respect to the energy
defined by the modulus of the reproducing kernel.

There are subsets of $\partial \mathbb B_d$ that have positive capacity according to our definition but capacity zero according to the definition with the modulus of the kernel.
Correspondingly, Corollary \ref{cor:cyclic_outer_cap} gives more information
than the analogous result for the modulus of the kernel.
A trivial example for $d=2$ is given by $f = z_2$.
Then $Z(f^*) = \{(z,0): z \in \partial \mathbb D\}$ is such a set; see Example \ref{exa:real_part_modulus}. Thus, Corollary \ref{cor:cyclic_outer_cap} applies to show that $f$ is not cyclic, whereas the analogous result for the modulus of the kernel does not apply.
Of course, it is also clear from the fact that $f$ has zeros inside ${\mathbb B_d}$ that $f$ is not cyclic.

In general if $p$ is a polynomial  that does not vanish inside $\mathbb{B}_d $, then we can write $Z(p)$ as a finite union of smooth manifolds \cite[Proposition 2.9.10]{Bochnak1998}. By definition the set $Z(p)$ is a (Z)-set and hence a (TN)-set in the sense of Rudin \cite[Definition 10.1.1]{Rudin08}. Then the smooth manifolds in the decomposition of 
 $Z(p)$ must be complex
tangential by \cite[Theorem 11.2.6]{Rudin08}.  But for
a set $E$ contained in a smooth complex tangential manifold, combining Theorems \ref{thm:boundary_limits_intro} and \cite[Theorem 2.8]{AC89} we find that $\ccap_{h^2_d}(E) = 0 $ if and only if the capacity generated by the modulus of the kernel of the Drury--Arveson space vanishes. 
We do not know if there exists a function $f\in H^2_d$ without zeros in $\mathbb B_d$ such that $\ccap_{h^2_d}(Z(f^*))>0$ but the capacity with the modulus of the kernel is zero. 
\end{remark}

Corollary \ref{cor:cyclic_outer_cap} also allows us to prove the existence of Kor\'anyi limits
for a class of functions that is larger than $H^2_d$, thus extending the scope of Theorem \ref{thm:boundary_limits_intro}. The \emph{Drury--Arveson--Smirnov class} is defined to be
\[
    N^+(H^2_d) = \Big\{ \frac{\varphi}{\psi}: \varphi, \psi \in \Mult(H^2_d), \psi \text{ cyclic} \Big\}.
\]
Background information on $N^+(H^2_d)$ can be found in \cite{AHM+17a,AHM+ar}.
In particular, $H^2_d \subset N^+(H^2_d)$.
\begin{corollary}
    For every $f \in N^+(H^2_d)$, there exists a Borel set $E \subset \partial \mathbb B_d$ with $\ccap_{h^2_d}^*(E) = 0$ such that $f$ has a Kor\'anyi limit et every point in $\partial \mathbb B_d \setminus E$.
\end{corollary}

\begin{proof}
    This is immediate from Theorem \ref{thm:h^2_d_boundary} and Corollary \ref{cor:cyclic_outer_cap}.
\end{proof}

There are versions of Theorem \ref{thm:inv_sub_quasi_everywhere} and Corollary \ref{cor:cyclic_outer_cap} for more general spaces $\mathcal{H}$ on the ball where we do not have access
to the weak-type inequality. But we need to settle for weaker statements, where outer capacity is replaced by inner capacity.
Thus, if $E \subset \partial \mathbb{B}_d$ and $f \in \mathcal{H}$, we say that $f^* = 0$
\emph{nearly everywhere} on $E$ if $\ccap_{\mathcal{H}}(E \setminus Z(f^*)) =0$.
(Property (P5) of Proposition \ref{prop:abstract_capacity} shows that this notion
is independent of the particular choice of boundary function $f^*$).
We let
\begin{equation*}
\mathcal{H}^E = \{ f \in \mathcal{H}: f^*= 0 \text{ nearly everywhere on } E \}.
\end{equation*}
It turns out that in the case of the Drury--Arveson space, if $E$ is Borel, then $(H^2_d)^E = (H^2_d)_E$
by the capacitability of Borel sets; this will be proved in Theorem \ref{thm:Choquet}.

For general spaces $\mathcal{H}$, it is no longer obvious that $\mathcal{H}^E$ is a subspace, since inner
capacity might fail to be subadditive in principle
(see however Lemma \ref{lem:capacity_zero}).
Nonetheless, it turns out to be a multiplier invariant subspace.

\begin{theorem}
  \label{thm:nearly_everywhere_subspace}
  Let $\mathcal{H}$ be a regular unitarily invariant space of holomorphic functions
  on $\mathbb{B}_d$ whose kernel has non-negative real part and let $E \subset \partial \mathbb{B}_d$.
  Then  $\mathcal{H}^E$ is a closed multiplier invariant subspace of $\mathcal{H}$.
\end{theorem}

\begin{proof}
  It is clear that $\mathcal{H}^E$ is invariant under multiplication by polynomials,
  so it suffices to show that $\mathcal{H}^E$ is a closed subspace of $\mathcal{H}$.
  To this end, we will show that
  \begin{equation}
    \label{eqn:nearly_everywhwere}
    \mathcal{H}^E = \bigcap \{ \ker(\rho_\mu): \mu \in P(\partial \mathbb{B}_d) \text{ with } \supp(\mu) \subset E \text{ and } \mathcal{E}(\mu) < \infty \}.
  \end{equation}
  Proposition \ref{prop:energy_fa} shows that each kernel on the right is closed,
  so this will establish the result.

  To prove \eqref{eqn:nearly_everywhwere},
  let $f \in \mathcal{H}^E$ and let $\mu \in P(\partial \mathbb{B}_d)$ be a measure with $\mathcal{E}(\mu) < \infty$
  whose support $K$ is contained in $E$. Then $K \setminus Z(f^*)$ is a Borel set
  whose inner capacity is zero, hence $\mu(K \setminus Z(f^*)) = 0$
  by Lemma \ref{lem:capacity_zero}.
  It follows that $f^* = 0$ almost everywhere with respect to $\mu$, thus
  Proposition \ref{prop:boundary_L^1} implies that
  \begin{equation*}
    \rho_\mu(f) = \int_{\partial \mathbb{B}_d} f^* \, d \mu = 0.
  \end{equation*}

  Conversely, suppose that $f \in \mathcal{H} \setminus \mathcal{H}^E$.
  Unravelling the definitions, we find a compact set
  $K \subset E \setminus Z(f^*)$ such that $\ccap(K) > 0$ and then
  $\nu \in P(K)$ with $\mathcal{E}(\nu) < \infty$.
  Since $f^*(\zeta) \neq 0$ for all $\zeta \in K$ and $\nu(K) = 1$,
  there exists a Borel subset $F \subset K$ such
  that $\int_{F} f^* \,d \nu \neq 0$.
  Then $d \mu = \nu(F)^{-1} \chi_F d \nu$ has finite energy (see Remark \ref{rem:energy_positive_kernel}).
  By Proposition \ref{prop:boundary_L^1}, we have
  \begin{equation*}
    \rho_\mu(f) = \int_{\partial \mathbb{B}_d} f^* \,d \mu = \nu(F)^{-1} \int_F f^* d \nu \neq 0,
  \end{equation*}
  showing that $f$ does not belong to the right-hand side of \eqref{eqn:nearly_everywhwere}.
\end{proof}

The following corollary
can be deduced from Theorem \ref{thm:nearly_everywhere_subspace} just as Corollary \ref{cor:cyclic_outer_cap}
was deduced from \ref{thm:inv_sub_quasi_everywhere}.
In fact, it can be proved directly using arguments along the lines of the discussion
at the beginning of the section.

\begin{corollary}
  \label{cor:cyclic_inner_cap}
  Let $\mathcal{H}$ be a regular unitarily invariant space of holomorphic functions
  on $\mathbb{B}_d$ whose kernel has non-negative real part.
  If $f \in \mathcal{H}$ is cyclic, then
  \begin{equation*}
    \ccap_{\mathcal{H}}( Z(f^*) ) = 0.
  \end{equation*}
\end{corollary}

\begin{proof}
  If $\ccap_{\mathcal{H}}(Z(f^*)) > 0$, then
  by definition, there exists a compact set $K \subset Z(f^*)$
  and $\mu \in P(K)$ with $\mathcal{E}(\mu) < \infty$. Then $\ker(\rho_\mu)$
  is closed by Proposition \ref{prop:energy_fa}.
  For every polynomial $p$, Proposition \ref{prop:boundary_L^1} shows that
  \begin{equation*}
    \rho_\mu(p f) = \int_{\partial \mathbb{B}_d} p f^* \,d \mu = 0,
  \end{equation*}
  so $[f] \subset \ker(\rho_\mu)$; whence $f$ is not cyclic.
\end{proof}

\section{Dual formulation and Choquet capacity}\label{sec:dual_formulation}

We show that our capacity can be expressed by a dual minimization problem.
\begin{theorem}\label{thm:dual_cap}
Let $\mathcal H$ be a unitarily invariant space
of pluriharmonic functions with reproducing kernel $K$ with non-negative real part.
Let $F\subset \bnd$ be a compact set. Then, 
\[ \ccap(F) = \inf \{ \Vert f \Vert_{\mathcal{H}}^2 : \Re f  \geq 1\,\, \text{on}\,\, F, \,\, f\in \mathcal{H} \cap C(\overline{\mathbb{B}_d}) \}. \]
\end{theorem}

\begin{proof}
If $\mu \in P(F)$ and $f \in \mathcal H \cap C(\overline{\mathbb B_d})$ with $\Re f \ge 1$ on $F$, then
Proposition \ref{prop:energy_fa} shows that
\begin{equation} 1 \leq \int_{\partial \mathbb{B}_d } \Re f \, d\mu \leq  \Big| \int_{\partial \mathbb{B}_d} f \, d\mu \Big| \leq  \Vert f \Vert_\mathcal{H} \mathcal{E}(\mu)^\frac12,  \end{equation}
which implies that $\ccap(F)$ is bounded above by the infimum in the statement.

Conversely, given $\varepsilon > 0$, we apply Lemma \ref{lem:c_r} to find $r \in (0,1)$ with $\ccap_r(F) < \ccap(F) + \varepsilon$. Let $\mu \in P(F)$ be an $r$-equilibrium measure for $F$ and consider the normalized potential
\[
  f(z) = \ccap_r(F) \int_{\partial \mathbb B_d}  K(rz, r w) \, d \mu(w).
\]
Then $f$ is continuous on $\overline{\mathbb B_d}$.
Lemma \ref{prop:potential_bound_below} shows that
$\Re f \ge 1$ on $F$ and $\|f\|_{\mathcal H}^2 \le \ccap_r(F)$;
hence the infimum is bounded above by $\ccap_r(F) < \ccap(F) + \epsilon$.
This finishes the proof.
\end{proof}

In the case of the Drury--Arveson space, Theorem \ref{thm:dual_cap} can be extended
to arbitrary subsets of $\partial \mathbb{B}_d$. Given $f \in \PHDA$,
Theorem \ref{thm:boundary_limits_intro} shows that there exists a Borel
set $E \subset \partial \mathbb{B}_d$ with $\ccap^*_{\PHDA}(E) = 0$
such that for all $\zeta \in \partial \mathbb{B}_d \setminus E$, the limit $f^*(\zeta) := \lim_{r \nearrow 1}f(r \zeta)$ exists. If $F \subset \partial \mathbb{B}_d$, we say that $\Re f^* \ge 1$ \emph{quasi-everywhere
on $F$} if
\begin{equation*}
  \ccap_{\PHDA}^*( \{ \zeta \in F: \Re f^* < 1 \} ) = 0.
\end{equation*}

\begin{lemma}
  \label{lem:qe_large_closed}
  Let $E \subset \partial \mathbb{B}_d$ be an arbirary subset and define
  \begin{equation*}
    C^E = \{ f \in \PHDA: \Re f^* \ge 1 \text{ quasi-everywhere on } E \}.
  \end{equation*}
  Then $C^E$ is a convex and weakly closed subset of $\PHDA$.
\end{lemma}

\begin{proof}
  Convexity of $C^E$ follows from the fact that the union of two sets of outer capacity zero has outer capacity zero;
  see Proposition \ref{prop:abstract_capacity}.
  We will show that $C^E$ is norm closed.
  Weak closedness of $C^E$ then follows from convexity and a standard consequence of the Hahn--Banach theorem.

  To establishes that $C^E$ is norm closed, we will use the weak-type inequality and argue similarly
  as in the proof of Theorem \ref{thm:inv_sub_quasi_everywhere}.
  Thus, let $(f_n)$ be a sequence in $C^E$ converging to $f \in \PHDA$.
  Let $t \in (0,1)$ and $\alpha > 1$. Then for each $n \in \mathbb{N}$,
  the weak-type inequality (Theorem \ref{thm:weak_type}) shows that
  \begin{align*}
    \ccap^*_{\PHDA}( E \cap \{\Re f^* < 1- t \})
    &\le
    \ccap^*_{\PHDA}( E \cap \{\Re (f_n^* - f^*) > t \}) \\
    &\le \ccap^*_{\PHDA} ( \{ M_\alpha[f_n - f] > t \}) \\
    &\lesssim \frac{\|f_n - f\|^2_{\PHDA}}{t^2}.
  \end{align*}
  Taking the limit $n \to 1$ and using that countable unions of sets of outer capacity
  zero have outer capacity zero (Proposition \ref{prop:abstract_capacity}), it follows
  that $\ccap^*_{\PHDA} ( E \cap \{\Re f^* < 1 \}) = 0$; whence $f \in C^E$.
\end{proof}

We now obtain the announced extension of Theorem \ref{thm:dual_cap} in the case of the Drury--Arveson space.

\begin{theorem}
  \label{thm:dual_DA}
  Let $E \subset \partial \mathbb{B}_d$ be an arbitrary subset. Then
  \begin{equation*}
    \ccap_{\PHDA}^*(E) = \inf \{ \|f\|_{\PHDA}^2: \Re f^* \ge 1 \text{ quasi-everywhere on } E \}.
  \end{equation*}
\end{theorem}

\begin{proof}
  For ease of notation, we will write $\|\cdot\| = \|\cdot\|_{\PHDA}$
  and $\ccap^* = \ccap^*_{\PHDA}$.
  To establish the inequality ``$\le$'', let $f \in \PHDA$ satisfy
  $\Re f^* \ge 1$ quasi-everywhere on $E$.
  We will show that $\ccap^*(E) \le \|f\|^2$.
  By replacing $f$ with $\Re f$, we may assume that $f$ is real-valued;
  observe that $\|\Re f \| = \| \frac{f + \overline{f}}{2}\|\le \|f\|$.
  Let $\alpha > 1$ and let $r \in [0,1)$. Since
  \begin{equation*}
    M_\alpha[f - f_r] + f_r \ge f^*
  \end{equation*}
  at every point $\zeta \in \partial \mathbb{B}_d$ at which the radial limit $f^*(\zeta) = \lim_{r \nearrow 1} f(r \zeta)$ exists, subadditivity of outer capacity yields for every $t \in (0,1)$ the bound
  \begin{align*}
    \ccap^*(E) &\le \ccap^* ( \{ M_{\alpha}[f - f_r] + f_r \ge 1 \}) \\
    &\le \ccap^*( \{ M_{\alpha}[f - f_r] > t \})
    +
   \ccap ( \{ f_r \ge 1 - t \}),
  \end{align*}
  where in the last step, we also used that inner and outer capacity agree
  for compact sets; see Proposition \ref{prop:abstract_capacity}.
  We estimate the first summand with the help of the weak-type estimate
  (Theorem \ref{thm:weak_type}) and the second summand with the help of Theorem \ref{thm:dual_cap} to obtain
  that
  \begin{equation*}
    \ccap^*(E) \le C_{\alpha} \frac{\|f - f_r\|^2}{t^2} + \frac{ \|f_r\|^2}{(1 - t)^2}.
  \end{equation*}
  Taking first the limit $r \to 1$ and then the limit $t \to 0$,
  we obtain the desired bound $\ccap^*(E) \le \|f\|^2$.

  For the reverse inequality, we first observe that the case
  when $E$ is compact is immediate from Theorem \ref{thm:dual_cap}.
  Next, we assume that $E$ is an open set $U$.
  We write $U$
  as an increasing union of compact sets $U = \bigcup_{n=1}^\infty K_n$.
  For each $n \in \mathbb{N}$, Theorem \ref{thm:dual_cap} yields a continuous
  function $f_n \in \PHDA$ with $\Re f_n \ge 1$ on $K_n$ and $\|f_n\|^2 \le \ccap(K_n) + \frac{1}{n}
  \le \ccap^*(U) + \frac{1}{n}$.
  Let $f \in \PHDA$ be a weak cluster point of the sequence $(f_n)$.
  Then $\|f\|^2 \le \ccap^*(U)$.
  Since $\Re f_n \ge 1$ on $K_k$ for all $n \ge k$,
  Lemma \ref{lem:qe_large_closed} implies that $\Re f^* \ge 1$ quasi-everywhere on $K_k$
  for all $k \in \mathbb{N}$.
  Since countable unions of sets of outer capacity zero have outer capacity zero
  (Proposition \ref{prop:abstract_capacity}),
  we find that $\Re f^* \ge 1$ quasi-everywhere on $U$. This concludes
  the proof of the remaining inequality for open sets.
  The case of an arbitrary set $E$ follows immediately from the definition of outer capacity.
\end{proof}

With the help of the last result, we can show that the outer capacity $\ccap^*_{\PHDA}$
is a Choquet capacity.

\begin{theorem}
  \label{thm:Choquet}
  The set function $\ccap^* = \ccap^*_{\PHDA}$, defined on all subsets of $\partial \mathbb{B}_d$,
  is a \emph{Choquet capacity}, i.e.\ the following properties hold:
  \begin{enumerate}[label=\normalfont{(\alph*)}]
    \item $\ccap^*(\emptyset) = 0$;
    \item if $E \subset F$ , then $\ccap^*(E) \le \ccap^*(F)$;
    \item if $(E_n)$ is a decreasing sequence of compact sets, then
      \begin{equation*}
        \ccap^*\Big( \bigcap_{n=1}^\infty E_n \Big)
        = \lim_{n \to \infty} \ccap^* (E_n);
      \end{equation*}
    \item if $(E_n)$ is an increasing sequence of arbitrary subsets of $\partial \mathbb{B}_d$, then
      \begin{equation*}
        \ccap^* \Big( \bigcup_{n=1}^\infty E_n \Big) = \lim_{n \to \infty} \ccap^*(E_n).
      \end{equation*}
  \end{enumerate}
  Therefore, all Borel subsets of $\partial \mathbb{B}_d$ are capacitible.
\end{theorem}

\begin{proof}
  All properties except for part (d) follow from general observations regarding abstract
  capacities; see Proposition \ref{prop:abstract_capacity}.
  As for (d), it suffices the prove the inequality
  \begin{equation}
    \label{eqn:Choquet_to_show}
    \ccap^* \Big( \bigcup_{n=1}^\infty E_n \Big) \le \lim_{n \to \infty} \ccap^*(E_n);
  \end{equation}
  note that the limit on the right exists and is finite.
  To this end, we use the dual description of outer capacity.
  Theorem \ref{thm:dual_DA} yields for each $n \in \mathbb{N}$
  a function $f_n \in \PHDA$ such that $\Re f_n \ge 1$ quasi-everywhere on $E_n$
  and $\|f_n\|^2 \le \ccap^*(E_n) + \frac{1}{n}$.
  Let $f \in \PHDA$ be a weak cluster point of the sequence $(f_n)$.
  Then $\|f\|^2 \le \lim_{n \to \infty} \ccap^*(E_n)$.
  Moreover, for each $k \in \mathbb{N}$, we have $\Re f_n \ge 1$ quasi-everywhere on $E_k$
  for all $n \ge k$, so Lemma \ref{lem:qe_large_closed} implies that $\Re f \ge 1$ quasi-everywhere
  on $E_k$. Since countable unions of sets of outer capacity zero have outer capacity zero,
  $\Re f \ge 1$ quasi-everywhere on $\bigcup_{k=1}^\infty E_k$.
  Another application of Theorem \ref{thm:dual_DA} then yields
  inequality \eqref{eqn:Choquet_to_show}, as desired.

  Capacitability of Borel sets finally follows from Choquet's theorem. See \cite[Chapter VI]{Choquet1955} or \cite[Theorem II:3]{Cegrell1988} for a more modern approach.
\end{proof}

\begin{remark}
  The proof of Theorem \ref{thm:Choquet} applies in any space in which a weak-type estimate
  for the radial maximal function holds.
  In particular, this includes the capacity $\ccap_{H^2_d}$ derived from the holomorphic Drury--Arveson kernel (equivalently, its real part) and the capacities associated with the spaces $\mathcal D_a$ for $a \in [0,1)$.
\end{remark}

\section{Totally null sets and Henkin measures}
\label{sec:TN_Henkin}

In this section, we connect our notion of capacity to the totally
null property and prove Theorem \ref{thm:cap_TN_intro}.
We again work in somewhat greater generality.
Let $\mathcal H$ be a regular unitarily invariant space of holomorphic
functions on $\mathbb B_d$ (see Section \ref{sec:cyclic}).

A complex Borel measure $\mu \in M(\partial \mathbb B_d)$ is said
to be \emph{$\Mult(\mathcal H)$-Henkin} if
whenever $(p_n)$ is a sequence of polynomials
with  $\sup_n \|p_n\|_{\Mult(\mathcal H)} < \infty$ and $\lim_{n \to \infty} p_n(z) = 0$ for all $z \in \mathbb B_d$, we have
\[
    \lim_{n \to \infty} \int_{\partial \mathbb B_d} p_n \, d \mu = 0.
\]
This is equivalent to demanding that the integration functional $p \mapsto \int_{\partial \mathbb B_d} p \, d \mu$, defined for (say) polynomials, extends to a weak-$*$ continuous
functional on $\Mult(\mathcal H)$
\footnote{In order to be consistent with Proposition \ref{prop:energy_fa},
it might be more natural to consider integration with respect to $\overline{\mu}$.
However, this leads to the same notion, since $\mu$ is Henkin if and only if $\overline{\mu}$ is Henkin;
see \cite[Lemma 3.3]{BHM18}.
}
; see \cite[Lemma 3.1]{BHM18}.
A Borel set $E \subset \partial \mathbb B_d$ is called \emph{$\Mult(\mathcal H)$-totally
null} if $|\mu|(E) = 0$ for every $\Mult(\mathcal H)$-Henkin measure $\mu$.

A regular unitarily invariant space of holomorphic functions on $\mathbb{B}_d$
is a \emph{complete Pick space} if $K$ does not vanish on $\mathbb{B}_d \times \mathbb{B}_d$
and the function $1 - \frac{1}{K}$ is positive semi-definite; see \cite{AM02} for background.
In particular, this implies that $\Re K \ge 0$.
The Drury--Arveson space in particular is a regular unitarily invariant complete Pick space of
holomorphic functions.

We say that a compact subset $E \subset \partial \mathbb{B}_d$ is an unboundedness for $\mathcal{H}$
if there exists $f \in \mathcal{H}$ with $\lim_{r \nearrow 1} |f(r \zeta)| = \infty$ 
for all $\zeta \in E$.
The equivalence of the totally null and the capacity zero condition now follows
fairly easily from what has already been done.
In particular, this will prove Theorem \ref{thm:cap_TN_intro}.

\begin{theorem}
  \label{thm:TN_capacity}
  Let $\mathcal{H}$ be a regular unitarily invariant complete Pick space of holomorphic functions on $\mathbb{B}_d$.
  For a compact set $E \subset \partial \mathbb{B}_d$, the following assertions are equivalent:
  \begin{enumerate}[label=\normalfont{(\roman*)}]
    \item $\ccap_\mathcal{H}(E) = 0$;
    \item $E$ is an unboundedness set for $\mathcal{H}$;
    \item $E$ is $\Mult(\mathcal{H})$-totally null.
  \end{enumerate}
\end{theorem}

\begin{proof}
  (i) $\Rightarrow$ (ii)
  This imiplication is a special case of Theorem \ref{thm:cap_0_unboundedness}.

  (ii) $\Rightarrow$ (iii) This implication was established in \cite[Theorem 4.1]{CH20}; it depends on the
  complete Pick property of $\mathcal{H}$.

  (iii) $\Rightarrow$ (i)
  The argument is essentially the same as that in \cite[Proposition 2.6]{DH23} or in \cite[Propostion 3.3]{CH20}.
  Indeed, if $\ccap_{\mathcal{H}}(E) > 0$, then $E$ supports a probability measure
  of finite energy. Proposition \ref{prop:energy_fa} implies that the integration functional $\rho_\mu$
  is bounded on $\mathcal{H}$. In particular,
  $\rho_\mu$ is weak-$*$ continuous on $\Mult(\mathcal{H})$,
  so  $\mu$ is $\Mult(\mathcal{H})$-Henkin, hence $E$ is not $\Mult(\mathcal{H})$-totally null.
\end{proof}

\begin{proof}[Proof of Theorem 1.3]
  By definition of the totally null property and regularity of Henkin measures,
  a Borel set $E \subset \partial \mathbb{B}_d$ is totally null if and only if each compact
  subset $F \subset E$ is totally null.
  On the other hand, capacitability of Borel sets (Theorem \ref{thm:Choquet} and the remark following
  it) shows that $E$ has capacity zero if and only if each compact subset does.
  Thus, Theorem \ref{thm:TN_capacity} shows that a Borel set is totally null
  if and only if it has capacity zero.
\end{proof}

\begin{remark}
  The implication (i) $\Rightarrow$ (iii) in Theorem \ref{thm:TN_capacity} may fail without the complete Pick property.
  For instance, if $\mathcal{H}$ is the weighted Bergman space on the unit disc
  with kernel $K_{\mathcal{D}_a}(z,w) = (1 - z \overline{w})^{-a}$ for $a > 2$, then $\Mult(\mathcal{H}) = H^\infty$,
  so the totally null subsets of $\partial \mathbb D$ are precisely the sets of Lebesgue measure zero.
  On the other hand, the energy of a positive measure is given by
  \begin{equation*}
    \mathcal{E}(\mu) \approx \sum_{n=0}^\infty |\widehat{\mu}(n)|^2 (n+1)^{a-1}.
  \end{equation*}
  By the Sobolev embedding theorem, every such measure has a continuous density.
  It follows that if a compact set $E \subset \partial \mathbb{D}$ has empty interior, then $\ccap_{\mathcal{H}}(E) = 0$.
  So if $E \subset \partial \mathbb{D}$ is any compact set with empty interior and positive Lebesgue measure,
  then $\ccap_{\mathcal H}(E) = 0$, but $E$ is not $\Mult(\mathcal H)$-totally null.
\end{remark}

\begin{remark}
  In \cite{CH20}, we also considered the notion of a \emph{weak unboundedness set $E$},
  meaning that there exist an auxiliary Hilbert space $\mathcal{E}$ and $f \in \mathcal{H} \otimes \mathcal{E}$
  with $\lim_{r \nearrow 1} \|f(r \zeta)\| = \infty$ for all $\zeta \in E$.
  It was shown in \cite[Theorem 4.1]{CH20} that for complete Pick spaces,
  the notions of unboudnedness  set, weak unboundedness set and totally null set agree.

  However, for the Bergman space, the unit circle is a weak unboundedness set that is not an unboundedness set.
  Indeed, the Privalov-Lusin radial uniqueness theorem (see \cite[\S IV.6]{Priwalow56}) shows that the circle is not an unboundedness set.
  On the other hand, the circle is a weak unboundedness set; see \cite{AD15}.
\end{remark}

The equivalence of the capacity zero and the totally null condition in Theorem \ref{thm:TN_capacity}
can be refined to a relationship between the classes of finite energy and Henkin measures.
We begin with an example.

\begin{example}
  \label{exa:Henkin_H2}
  We consider the Hardy space $H^2$ on the disc.
  It is a consequence of the classical F.\ and M.\ Riesz theorem that
  the Henkin measures for $H^2$ are precisely those measures on $\partial \mathbb{D}$
  that are absolutly continuous with respect to Lebesgue measure $m$;
  see for instance \cite[Remark 9.2.2 (c)]{Rudin08}.

  On the other hand, we claim that a measure $\mu \in M(\partial \mathbb{D})$ has finite energy for $H^2$ if and only if
  $d \mu = \overline{f} dm + g d m$ for some $f \in H^1_0$ and $g \in H^2$.
  Indeed, sufficiency follows from the fact that integration against ${f} dm$ annihilates $H^2$ and Proposition \ref{prop:energy_fa}. As for necessity, suppose that $\mu$ has finite energy
  and let $g = f_\mu$ be the holomorphic potential of $\mu$. Then Proposition \ref{prop:energy_fa}
  shows that $g \in H^2$ and that
  \begin{equation*}
    \int_{\partial \mathbb{D}} p \, d \overline{\mu} = \int_{\partial \mathbb{D}} p \overline{g} \, d m
  \end{equation*}
  for all polynomials $p$. The F.\ and M.\ Riesz theorem
  shows that $d \overline{\mu} - \overline{g} dm = f d m$ for some $f \in H^1_0$,
  so $\mu$ has the stated form.
\end{example}

Proposition \ref{prop:energy_fa} shows that every complex measure of finite energy induces a continuous
functional on $\mathcal{H}$. In particular, every finite energy measure is Henkin.
As Example \ref{exa:Henkin_H2} shows, the converse need not hold.
Example \ref{exa:Henkin_H2} shows another important distinction
between the two classes of measures:
While the set of Henkin measures forms a band (see \cite[Lemma 3.3]{BHM18}), the set of finite energy measures
is in general not even closed under complex conjugation.

\begin{remark}
  It is a remarkable property of the Dirichlet space $\mathcal{D}$ that unlike in the case of $H^2$, the complex
  conjugate of a finite energy measure also has finite energy;
  see \cite{Koosis83} or \cite[p.\ 186, Lemma 8]{Wojtaszczyk91}.
\end{remark}

The next result in particular establishes Theorem \ref{thm:henkin_intro}.

\begin{theorem}
  \label{thm:Henkin_char}
  Let $\mathcal{H}$ be a regular unitarily invariant complete Pick space of holomorphic functions on $\mathbb{B}_d$.
  A measure $\mu \in M(\partial \mathbb{B}_d)$ is $\Mult(\mathcal{H})$-Henkin
  if and only if it is absolutely continuous with respect to a probability measure $\mu$ that has finite $\mathcal{H}$-energy.
\end{theorem}

\begin{proof}
  As remarked in the discussion following Example \ref{exa:Henkin_H2}, every measure with finite energy is Henkin.
  Since the Henkin measures form a band by \cite[Lemma 3.3]{BHM18}, so is every measure that is absolutely continuous with respect to a finite energy measure.

  Conversely, suppose that $\mu$ is Henkin. Let
  \begin{equation*}
    \Lambda = \{ \rho \in M_+(\partial \mathbb{B}_d): \mathcal{E}(\mu) \le 1 \}.
  \end{equation*}
  Then $\Lambda$ is convex, for instance by Proposition \ref{prop:energy_fa} and the triangle inequality.
  Since $\mu(\partial \mathbb{B}_d) = \int_{\partial \mathbb{B}_d} 1 \,d \mu \le \mathcal{E}(\mu)$,
  we see that every measure in $\Lambda$ has total mass at most $1$.
  By lower semicontinuity of the energy functional, $\Lambda$ is weak* closed and hence weak* compact.
  (The weak* topology is metrizable on bounded subsets of $M(\partial \mathbb{B}_d)$,
  so it suffices to check sequential closedness).

  In this setting, the Glicksberg--K\"onig--Seever decomposition theorem
  yields a decomposition $\mu = \mu_a + \mu_s$, where $\mu_a$ is absolutely
  continuous with respect to some $\rho \in \Lambda$ and $\mu_s$ is concentrated
  on a Borel set $E$ that is null for every measure in $\Lambda$;
  see for instance \cite[Corollary II.7.5]{Gamelin69} and also \cite[Theorem 9.4.4]{Rudin08}
  for a slightly weaker statement.
  For every compact set $K \subset E$, it follows from the definition
  that $\ccap(K) = 0$, hence $K$ is totally null by Theorem \ref{thm:TN_capacity}.
  Regularity of the measures involved shows that $E$ is totally null.

  Now, by the already established implication of the theorem, we find that $\mu_a$ is Henkin,
  hence so is $\mu_s = \mu - \mu_a$. Since $\mu_s$ is concentrated on the totally
  null set $E$, it follows that $\mu_s = 0$. Therefore, $\mu = \mu_a$ is absolutely
  continuous with respect to some $\rho \in \Lambda$. Rescaling $\rho$, we obtained
  the desired probability measure.
\end{proof}

A measure $\mu \in M(\partial \mathbb B_d)$ is said to be \emph{$\Mult(\mathcal H)$-totally
singular} if it is singular with respect to every $\Mult(\mathcal H)$-Henkin measure.
Such measures appear in the description of the dual space of certain
algebras of multipliers; see \cite{DH23}.
Our results also give a characterization of totally singular measures.

\begin{theorem}
\label{thm:TS}
  Let $\mathcal{H}$ be a regular unitarily invariant complete Pick space of holomorphic functions on $\mathbb{B}_d$.
  The following are equivalent for a measure $\nu \in M(\partial \mathbb{B}_d)$:
  \begin{enumerate}[label=\normalfont{(\roman*)}]
      \item $\nu$ is $\Mult(\mathcal H)$-totally singular;
      \item $\nu$ is singular with respect to every probability measure
      of finite $\mathcal H$-energy;
      \item $\nu$ is concentrated on an $F_\sigma$ set of inner $\mathcal H$-capacity zero.
  \end{enumerate}
\end{theorem}

\begin{proof}
    The equivalence of (i) and (ii) is immediate from the definition and Theorem \ref{thm:Henkin_char}.

    (i) $\Rightarrow$ (iii) Let $\nu$ be totally singular.
    According to \cite[Proposition 4.4]{DH23}, $\nu$ is concentrated on an $F_\sigma$
    set $E$ that is $\Mult(\mathcal H)$-totally null. By the easy implication
    of Theorem \ref{thm:TN_capacity}, every compact subset of $E$ has $\mathcal H$-capacity
    zero, so $\ccap_{\mathcal H}(E) = 0$.

    (iii) $\Rightarrow$ (i) Let $\nu$ be concentrated on an $F_\sigma$ set
    $E$ satisfying $\ccap_{\mathcal H}(E) = 0$. It follows from Theorem \ref{thm:TN_capacity}
    that $E$ is $\Mult(\mathcal H)$-totally null. By definition,
    $\nu$ is therefore singular with respect to every $\Mult(\mathcal H)$-Henkin measure.
\end{proof}

\appendix
\section{Abstract capacities}
\label{sec:app_abstract_capacities}

In this section we shall develop some properties of abstract capacities that essentially depend only on the lower semicontinuity of the energy functional. Most of the proofs are quite standard but usually are not carried out in this generality.
We mostly follow \cite[Chapter 2]{EKM+14}.
Throughout, let $X$ be a compact metric space and $ M_+(X)$ the set of positive and finite Borel measures on $X$. Let also $C \subset M_+(X)$ be a convex cone.

\begin{definition} \label{defn:energy_functional}
Let $C \subset M_+(X)$ be a convex cone such that for all $\mu \in C$ and all compact $F \subset X$, the measure $\chi_F d \mu$ belongs to $C$.
A function $I:C \times C \to [0,\infty)$ is called an energy functional if 
    \begin{itemize}
        \item[(A1)]  $I(\mu,\nu) = I(\nu,\mu)$ if $\mu,\nu \in C$.
        
        \item[(A2)] $I(t \mu + \rho,\nu) = t I(\mu,\nu) + I(\rho,\nu) $ if $\mu,\nu,\rho \in C $ and $ t \in [0,\infty)$.

        \item[(A3)] 

        If $\mu_n \in C$ and $\mu_n$ converges weak* to $\mu$,  then $\liminf_{n} I(\mu_n,\mu_n) \geq I(\mu,\mu)$ if $\mu \in C$ and $\lim_n I(\mu_n,\mu_n) = \infty$ if $\mu \not \in C$.
    \end{itemize}

    \begin{remark}
       We will write $I(\mu)$ instead of $I(\mu,\mu)$. We shall refer to measures $\mu \in C $  as measures of finite energy. In view of $(A3)$ it makes sense to define $I(\mu,\mu) = \infty$ for all $\mu \in M_+(X)\setminus C $. Then $(A3)$ can be restated simply as saying that $I(\cdot,\cdot)$ is sequentially weak* lower semicontinuous on $M_+(X).$
    \end{remark}

Given an energy functional we can define the capacity of compact sets as
\[ \ccap(F)=\sup\bigg\{ \frac{1}{I(\mu)} : \mu \in P(F)\bigg\}. \]
In particular, a compact set has positive capacity if and only if it supports a probability
measure of finite energy.

Given a compact set $F$, a probability measure $\mu\in P(F)$ is called an equilibrium measure if $I(\mu) = \ccap(F)^{-1}$. The existence of equilibrium measures for all compact sets of positive capacity is an immediate consequence of (A3) and weak* sequential
compactness of $P(F)$.

We can extend the definition of capacity to all subsets of $X$ by approximating from within by compact sets. That is, if $E\subset X$, we define
\[  \ccap(E)=\sup\{  \ccap(F): F\subset E, F \,\, \text{compact}\, \}. \]
We shall refer to this capacity as inner capacity. 

Similarly, an outer capacity can be defined by approximating from outside by open sets 
\[ \ccap^*(E) = \inf\{ \ccap(G): G \supset E, G \,\, \text{open} \, \}. \]

A set will be called capacitable if its inner and outer capacities agree. By definition open sets are capacitable. 
\end{definition}

\begin{proposition}\label{prop:abstract_capacity}
    Let $I$ be an energy functional.  The following properties hold. 
    \begin{itemize}
        \item [(P1)] If $E_1 \subset E_2$ then $\ccap^*(E_1) \leq \ccap^*(E_2)$.
        \item[(P2)] If $E_1, E_2, \dots \subset X$, then $\ccap^* \Big( \bigcup_{n=1}^\infty E_n \Big) \leq \sum_{n=1}^\infty \ccap^*(E_n)$.
        \item[(P3)] If $F_1 \supset F_2, \dots$ is a decreasing sequence of compact sets \[ \ccap \Big( \bigcap_{n=1}^\infty F_n \Big) = \lim_n \ccap(F_n).\]
        \item[(P4)] Open and closed sets are capacitable. 
        \item[(P5)] If $E_1,E_2 \subset X$, then $\ccap(E_1 \cup E_2) \le \ccap(E_1) + \ccap^*(E_2)$.
        \item[(P6)] If $F \subset X$ is compact with positive capacity and $\mu$ is an equilibrium measure for $F$, then for every $\nu \in P(F)$ with finite energy we have
        \[ I(\mu) \leq I(\mu,\nu). \]
    \end{itemize}

\begin{proof}
    $(P1)$ is an immediate consequence of the definitions, while $(P2)$ is proved exactly as the analogous statement for the outer capacity in \cite[Theorem 2.1.9]{EKM+14}.
    Here, we use that $C$ is closed under restrictions to compact subsets of $F$.

      In order to prove $(P3)$, let $(F_n)$ be a decreasing sequence of compact sets.
      We may assume that $\ccap(F_n) > 0$ for all $n$. Let $\mu_n$ be an equilibrium measure for $F_n$. We pick a weak* cluster point $\mu$ of $(\mu_n)$. Since $\supp(\mu)  \subset F_n $ and $F_n$ are decreasing, $\supp(\mu) \subset \bigcap_{n=1}^\infty F_n $. Hence by the lower semicontinuity of the energy functional we have that 

\[ \ccap\Big( \bigcap_{n=1}^\infty F_n \Big)^{-1} \leq I(\mu) \leq \liminf_n I(\mu_n) = \lim_n \ccap(F_n)^{-1}. \]
The other inequality is clear by inclusions. 

Open sets are capacitable by definition. The capacitability of closed sets is a consequence of $(P3)$. Indeed, let $F$ be a compact set. There exists  sequence $(G_n)$ of open sets such that $F \subset \overline{G_{n+1}} \subset G_n$ for all $n$ and $\bigcap_{n} \overline{G_n} = F$. Then by definition of outer capacity and $(P3)$ we have that 
\[ \ccap(F) \leq \ccap^*(F) \leq \lim_n\ccap(G_n) \leq \lim_n\ccap(\overline{G_n}) = \ccap\Big( \bigcap_{n=1}^\infty \overline{G_n}\Big)=\ccap(F).  \]

To show $(P5)$, let $F \subset E_1 \cup E_2$ be compact and let $G \supset E_2$ be open.
Since $F = (F \setminus G) \cup (F \cap G)$ and $F \setminus G$ is a compact subset of $E_1$, (P2) and (P4) imply
\begin{align*}
  \ccap(F) \le \ccap^*(F \setminus G) + \ccap^*(F \cap G)
  &= \ccap(F \setminus G) + \ccap^*(F \cap G) \\
  \le \ccap(E_1) + \ccap^*(G).
\end{align*}
Now $(P5)$ follows by taking the infimum over all open $G \supset E_2$ and the supremum over all compact $F \subset E_1 \cup E_2$.

Finally $(P6)$ follows by a variational argument.  For $t\in (0,1)$ consider the measure  $(1-t) \mu + t \nu \in P(F)$. By definition of equilibrium measure, 
\[ I(\mu) \le I( (1-t) \mu + t \nu) = I(\mu) +2t(I(\mu,\nu) - I(\mu)) + t^2(I(\mu)+I(\nu)-2I(\mu,\nu)), \]
hence the coefficient of the first order term must be non-negative.
\end{proof}
\end{proposition}

\bibliographystyle{amsplain}
\bibliography{DA_potential_theory_refs}

\end{document}